\documentclass[11pt,halfline,a4paper]{article}

\usepackage[utf8]{inputenc}
\usepackage[english]{babel}
\usepackage{amsmath}
\usepackage{amsthm}
\usepackage{amsfonts}
\usepackage{amssymb}
\usepackage[pdftex]{color,graphicx}
\usepackage[titletoc]{appendix}
\usepackage{mathrsfs}  
\usepackage[margin=2cm]{caption}
\usepackage{tikz-cd}
\usepackage{pxfonts}
\usepackage{bbold}
\usepackage{setspace}
\usepackage{subcaption}

\usepackage[left=1.8cm,right=1.8cm,top=2.3cm,bottom=2.4cm]{geometry}

\newtheorem{lemme}{Lemma}
\newtheorem{prop}{Proposition}

\newtheorem{theorem}{Theorem}
\newtheorem{Theo}{Theorem}

\newtheorem{corollary}{Corollary}

\newtheorem{remark}{Remark}

\newcommand\E{ \mathbf{E} }

\newcommand\PP{ \mathbf{P} }

\title{Spectral Measures of Spiked Random Matrices}
 \author{Nathan Noiry}
\date{}

\begin{document}
\maketitle

\abstract{We study two spiked models of random matrices under general frameworks corresponding respectively to additive deformation of random symmetric matrices and multiplicative perturbation of random covariance matrices. In both cases, the limiting spectral measure in the direction of an eigenvector of the perturbation leads to old and new results on the coordinates of eigenvectors.}

\vspace{0.3cm}

\section{Introduction}

The study of deformed models of random matrices has been the subject of tremendous amount of works in the last decades. In this paper, we study two of them. The first one corresponds to an additive perturbation of a symmetric random matrix (or Wigner matrix):
\[  W_n = \frac{1}{\sqrt{n}}X_n + A_n \]
where
\begin{itemize}
\item $X_n$ is a random symmetric matrix of size $n \times n$, whose entries are, up to symmetry, i.i.d. centered and reduced;
\item $A_n$ is a deterministic symmetric matrix of size $n \times n$ (or random, independent of $X_n$).
\end{itemize}
The second one is a multiplicative deformation of a random covariance matrix (or Wishart matrix):
\[ S_n = \frac{1}{n} \Sigma_n^{1/2} X_n X_n^T \Sigma_n^{1/2},  \]
where
\begin{itemize}
\item $X_n$ is a random matrix of size $n \times n$, $m/n \rightarrow \alpha \in (0,\infty)$, whose entries are i.i.d. centered and reduced;
\item $\Sigma_n$ is a deterministic symmetric matrix (or random, independent of $X_n$) having non-negative eigenvalues.
\end{itemize}
The spectra of these models have been well studied. Let $\mathrm{Spec}(W_n)$ (resp. $\mathrm{Spec}(S_n)$) be the set of eigenvalues of $W_n$ (resp. $S_n$) counted with multiplicities. The empirical spectral measures of $W_n$ and $S_n$ are:
\[  \mu_{W_n} = \frac{1}{n} \sum\limits_{ \lambda \in \mathrm{Spec}(W_n) } \delta_{\lambda} \quad \text{and} \quad \mu_{S_n} = \frac{1}{n} \sum\limits_{ \lambda \in \mathrm{Spec}(S_n) } \delta_{\lambda} . \]
Under mild assumptions, they converge respectively towards probability measures $\mu_{sc} \boxplus \mu_A$ and $\mu_{\mathrm{MP},\alpha} \boxtimes \mu_\Sigma$ properly defined in Subsections \ref{subsec: Wigner general framework} and \ref{subsec: Wishart general framework}.

In both models, we define an outlier as an eigenvalue that does not lie in the neighborhood of the support of the limiting spectrum. Many works identified necessary and sufficient conditions on the spectrum of $A_n$ (resp. $\Sigma_n$) for the appearance of outlier in the spectrum of $W_n$ (resp. $S_n$). The seminal paper is due to Baik, Ben Arous and P\'ech\'e \cite{MR2165575}, who identified a phase transition for the existence of an outlier in the covariance setting with $\Sigma_n = \mathrm{Diag}(\theta,1, \ldots,1)$, $\theta \geq 1$. The most recent results can be found in \cite{MR3729610} and we refer to the survey \cite{MR3792626} for an extensive bibliography. In all previous approaches, two main techniques were used: a clever identity on determinants (first remarked in \cite{MR2782201}), and a precise analysis of the empirical spectral measure.

The goal of this paper is to bring into focus the possible use of the spectral measures in the study of deformed models of random matrices. Let us introduce them. Let $\theta$ be an eigenvalue of $A_n$ (or $\Sigma_n$) which we consider as atypical in that it may be responsible for the existence of an outlier. Denote $v_1^{(n)}$ the associated eigenvector. We will call $\theta$ a spike of $W_n$ (resp. $S_n$) and $v_1^{(n)}$ the direction of the spike. The spectral measures in the direction of the spike are respectively defined by:
\[ \mu_{(W_n,v_1^{(n)})} := \sum\limits_{ \lambda \in \mathrm{Spec}(W_n) } \left| \langle \phi_{\lambda} , v_1^{(n)}   \rangle   \right|^2 \delta_\lambda \quad \text{and} \quad  \mu_{(S_n,v_1^{(n)})} := \sum\limits_{ \lambda \in \mathrm{Spec}(S_n) } \left| \langle \phi_{\lambda} , v_1^{(n)}   \rangle   \right|^2 \delta_\lambda,  \]
where $\phi_\lambda$ is a normalized eigenvector associated to eigenvalue $\lambda$. Note that unlike empirical spectral measures, these probability measures contain information on the eigenvectors of $W_n$ and $S_n$. Following the well-known observation that outliers have associated eigenvectors which are localized in the direction of the spike, their influence should be present in $\mu_{(W_n,v_1^{(n)})}$ and $\mu_{(S_n,v_1^{(n)})}$ at a macroscopic level. Moreover, they can be easily studied as their Stieltjes transforms are given by the generalized entries of the resolvent ($\langle v_1^{(n)}, (W_n-z)^{-1} v_1^{(n)} \rangle \text{ and }  \langle v_1^{(n)}, (S_n-z)^{-1} v_1^{(n)} \rangle$), for which many results already exist. In particular, as stated in Corollaries \ref{coro: Wigner case, as convergence} and \ref{coro: Wishart general case as convergence}, $\mu_{(W_n,v_1^{(n)})}$ and $\mu_{(S_n,v_1^{(n)})}$ converge weakly towards deterministic probability measures denoted $\mu_{sc,A,\theta}$ and $\mu_{\alpha,\Sigma,\theta}$. We are going to present two applications of the spectral measures.

The first one recovers a classical result concerning the value of an outlier and the norm of its associated eigenvector projection in the direction of the spike.

The second one is concerned with the behavior of the projection of non-outlier eigenvectors in the direction of the spike. Namely, in the setting of an additive perturbation, if $f_{sc,A}$ and $f_{sc,A,\theta}$ are the respective densities of $\mu_{sc} \boxplus \mu_A$ and $\mu_{sc,A,\theta}$ and if $x$ is in the support of $\mu_{sc} \boxplus \mu_A$, we prove the following convergence in probability
\begin{equation}  \frac{n}{ \left| \left\{ \lambda \in \mathrm{Spec}(W_n), \, |\lambda - x| \leq \varepsilon_n  \right\}  \right| } \sum\limits_{ \lambda \in \mathrm{Spec}(W_n), \, |\lambda -x| \leq \varepsilon_n } \left|   \left\langle \phi_\lambda, v_1^{(n)}   \right\rangle \right|^2  { \underset{ n \rightarrow + \infty}{ \longrightarrow} }  \frac{f_{sc,A,\theta}(x)}{f_{sc,A}(x)}, 
\label{eq: intro eq proj non-outlier}
\end{equation}
for any sequence $n^{-1/2} < \! \! < \varepsilon_n < \! \! < 1$. In other words, the left-hand side, which is an average in the vicinity of $x$ of the square-projections of eigenvectors in the direction of the spike, converges towards a deterministic profile. A similar result holds in the covariance setting. These are the content of Theorems \ref{theo: square proj Wigner general case} and \ref{theo: square proj Wishart general case}. Our proof is inspired by the work of Benaych-Georges, Enriquez and Micha\"il \cite{benaych2018eigenvectors} and uses local laws estimates recently obtained by Knowles and Yin in \cite{MR3262497}. When $\theta$ belongs to the support of the asymptotic spectrum of $A_n$ (resp. $\Sigma_n$), Theorems \ref{theo: square proj Wigner general case} and \ref{theo: square proj Wishart general case} are {\it microscopic} confirmations of the results of Allez and Bouchaud \cite{MR3256861} (in the Wigner setting) and Ledoit and Péché \cite{MR2834718} (in the Wishart setting), who derived the asymptotic behavior of the overlaps $| \langle \phi_\lambda, v_\gamma \rangle|^2$ by taking the average over eigenvectors $\phi_\lambda$ associated to eigenvalues of $W_n$ (resp. $S_n$) belonging to a {\it macroscopic} part of $\mathrm{Supp}\left( \mu_{sc} \boxplus \mu_A \right)$ (resp. $\mathrm{Supp}\left(\mu_{\mathrm{MP},\alpha} \boxtimes \mu_\Sigma \right)$) and the average over eigenvectors $v_\gamma$ of $A_n$ (resp. $\Sigma_n$) belonging to a {\it macroscopic} part of the asymptotic spectrum of $A_n$ (resp. $\Sigma_n$). When $\theta$ does not belong to the support of the asymptotic spectrum of $A_n$ (resp. $S_n$), such a macroscopic average is not available whereas the spectral measure approach still works.

Interestingly, in rank one perturbation cases, that is when $A_n$ (resp. $\Sigma_n$) has only one nonzero (resp. non one) eigenvalue $\theta$, all the computations are explicit. This is because $\mu_{sc} \boxplus \mu_A = \mu_{sc}$ and $\mu_{\mathrm{MP},\alpha} \boxtimes \mu_{\Sigma} = \mu_{\mathrm{MP},\alpha}$ are explicit, respectively equal to the semicircle law and to the Marchenko-Pastur law with parameter $\alpha$:
\begin{align*}
\mu_{sc}(\mathrm{dx}) &= \frac{\sqrt{4-x^2}}{2 \pi} \mathbf{1}_{|x| \leq 2} \mathrm{d}x \\
\mu_{\mathrm{MP},\alpha}(\mathrm{d}x) &= \frac{\sqrt{(b-x)(x-a)}}{2 \pi x} \mathbf{1}_{(a,b)}(x) \mathrm{d}x  + \mathbf{1}_{\alpha<1} (1-\alpha) \delta_0(\mathrm{d}x),
\end{align*}
where $a,b = (1 \pm \sqrt{\alpha})^2$. In particular, we obtain the following formulas for the spectral measures in the direction of the spike:
\begin{align*}
\mu_{sc,\theta}(\mathrm{d}x) &= \frac{\sqrt{4-x^2}}{2 \pi (\theta^2 + 1 - \theta x)} \mathbf{1}_{|x| \leq 2} \mathrm{d}x + \mathbf{1}_{|\theta|>1} \left( 1 - \frac{1}{\theta^2}  \right) \delta_{ \theta + \frac{1}{\theta }} (\mathrm{d}x)  \\
\mu_{\mathrm{MP},\alpha,\theta}(\mathrm{d}x) &= \frac{\theta \sqrt{(b-x)(x-a)}}{2 \pi x \left( x(1-\theta) + \theta( \alpha \theta - \alpha + 1) \right)} \mathbf{1}_{(a,b)}(x) \mathrm{d}x + c_{\alpha,\theta}\delta_0(\mathrm{d}x) + d_{\alpha,\theta} \mathbf{1}_{|\theta-1|> \frac{1}{\sqrt{\alpha}} } \delta_{ x_{\alpha,\theta} } (\mathrm{d}x),
\end{align*}
where $c_{\alpha,\theta}$, $d_{\alpha,\theta}$ and $x_{\alpha,\theta}$ are explicit constants, see Propositions \ref{prop: Wigner case explicit comput of limiting measure} and \ref{prop: Wishart case explicit comput of limiting measure}. These two measures belong to the class of the so-called {\it free Meixner} laws which appear in a Gaussian context in the work of Lenczewski \cite{MR3373058}.

Hence, by \eqref{eq: intro eq proj non-outlier}, the limiting profiles for the averaged square projections of non-outlier eigenvectors are also explicit. We give numerical simulations that agree with our predictions, see Subsections \ref{subsec: Wigner explicit} and \ref{subsec: Wishart explicit}. In the covariance setting, Bloemendal, Knowles, Yau and Yin proved that individual square projections of non-outlier eigenvectors that  are associated to eigenvalues in the vicinity of the edge (of $b$) converge towards a chi squared random variable with given variance (see \cite[Theorem 2.20]{MR3449395}). Although it requires an averaging step, our result completes the picture as it is concerned with eigenvectors associated to any fixed location of the bulk of the spectrum. We believe that the convergence still holds for smaller averaging windows and provide numerical simulations supporting this conjecture at the end of Section \ref{sec: proof averaged}.

Let us finally say a few words about previous use of spectral measures in the literature. Benaych-Georges, Enriquez and Micha\"il, in \cite{benaych2018eigenvectors}, obtained informations on the eigenvectors of a diagonal deterministic matrix perturbed by a random symmetric matrix. In a series of works (the most recent being \cite{gamboa2018sum}), Gamboa, Nagel and Rouault studied spectral measures of some classical ensembles of random matrix theory and their connections with sum rules. More related to our setting, in \cite{MR2330979}, Bai, Miao and Pan studied the spectral measure at the first vector of the canonical basis $e_1$ in general covariance cases, in the absence of spike.

We emphasize that our method could apply to other deformed models such as multiplicative perturbation of Wigner matrices or information plus noise matrices, but we chose to restrict our scope so that the present paper remains short and comprehensive.

\paragraph*{Notations and organization of the paper.}
The random matrices that we study are built from an i.i.d. collection of real random variables $(X^{(n)}_{ij})$, $n \geq 1$, $1 \leq i,j \leq n$. Let $X$ be a generic random variable with the same law. We suppose that $\E[X]=0$, $\E[X^2]=1$ and that $X$ has moments of all order. Notice that the complex case could also be treated, replacing each transposed matrix $A^T$ by its transposed-conjugate $A^*$, and making the hypothesis that $\E[|X|^2]=1$.

For a complex $z \in \mathbf{C}$, we will denote $\Re(z)$ and $\Im(z)$ the real part and imaginary part of $z$.

For a probability measure $\nu$, we always denote
\[  s_{\nu}(z) := \int_{\mathbf{R}} \frac{\mathrm{d}\nu(x)}{x-z}  \]
its Stieltjes transform that maps the upper half-plane to itself.

In Subsections \ref{subsec: Wigner general framework} and \ref{subsec: Wishart general framework} we give general results concerning additive perturbation of a Wigner matrix and multiplicative perturbation of a Wishart matrix. Subsections \ref{subsec: Wigner explicit} and \ref{subsec: Wishart explicit} provide explicit computations for rank-one deformation cases. The proofs are done in Sections \ref{sec: proof as convergence} and \ref{sec: proof averaged}.

\section{Additive perturbation of a Wigner matrix}\label{sec: Wigner general}
\subsection{The general framework}\label{subsec: Wigner general framework}
In this section we consider for each $n \geq 1$ the following Wigner matrix:
\[ X_n := \left[
\begin{array}{cccc}
   X^{(n)}_{11} & X^{(n)}_{12} & \cdots & X^{(n)}_{1n} \\
   X^{(n)}_{12} & X^{(n)}_{22} & \cdots & X^{(n)}_{2n} \\
   \vdots       & \vdots       & \ddots & \vdots  \\
   X^{(n)}_{1n} & X^{(n)}_{2n} & \cdots & X^{(n)}_{nn}
\end{array}
\right].  \]
We also consider $A_n$ a deterministic matrix (or random matrix independent of $X_n$) whose eigenvalues are $\gamma^{(n)}_1=\theta, \gamma^{(n)}_2, \ldots, \gamma^{(n)}_n$, with associated eigenvectors $v^{(n)}_1, \ldots, v^{(n)}_n$. We suppose that there exists a probability measure $\mu_{A}$ such that
\begin{equation}  \frac{1}{n}\sum\limits_{i=1}^n \delta_{\gamma^{(n)}_i} \underset{ n \rightarrow + \infty}{ \longrightarrow} \mu_A 
\label{eq: assumption spec perturbation (add) }
\end{equation}
in the sense of weak convergence. Moreover, let us assume that there exists $\Gamma > 0$ such that, for all $n \geq 1$, $\sup_{1 \leq i \leq n} |\gamma_i^{(n)}| \leq \Gamma$, namely that the eigenvalues of the perturbation remain bounded.

We will study following additive perturbation model:
\[  W_n := \frac{1}{\sqrt{n}}X_n + A_n.  \] 

Let $\lambda^{(n)}_1 \geq \cdots \geq \lambda^{(n)}_n$ be the eigenvalues of $W_n$ and $\phi^{(n)}_1, \ldots, \phi^{(n)}_n$ the associated normalized eigenvectors. Under assumption \eqref{eq: assumption spec perturbation (add) }, it is known that the empirical spectral measure $\mu_{W_n}$ converges towards a deterministic probability measure which is the free convolution $\mu_{sc} \boxplus \mu_A$ between the semicircle law $\mu_{sc}(\mathrm{d}x) = (2 \pi)^{-1} \sqrt{4-x^2} \mathbf{1}_{|x| \leq 2} \mathrm{d}x$ and $\mu_A$. Its Stieltjes transform is characterized by
\begin{equation}  s_{\mu_{sc} \boxplus \mu_A}(z) = \int_{\mathbf{R}} \frac{d\mu_A(\lambda)}{\lambda - s_{\mu_{sc} \boxplus \mu_A}(z) - z} . 
\label{eq: free convolution def}
\end{equation}
This has first been shown by Pastur in \cite{MR0475502}. The study of \eqref{eq: free convolution def} provides information on the probability measure $\mu_{sc} \boxplus \mu_A$. In particular, Biane proved that it has a smooth density with respect to the Lebesgue measure, see \cite{MR1488333}.

The parameter $\theta$ is considered as a spike which may create an outlier in the spectrum, that is an eigenvalue that does not lie in $\mathrm{Supp}( \mu_{sc} \boxplus \mu_A)$. Following the heuristic that an outlier in the spectrum creates a localized eigenvector, we study the spectral measure in the direction of the spike:

\[ \mu_{(W_n,v^{(n)}_1)} = \sum\limits_{i=1}^n |  \langle \phi^{(n)}_i, v^{(n)}_1 \rangle |^2 \delta_{ \lambda^{(n)}_i }.   \]

The Stieltjes transform of $\mu_{(W_n,v^{(n)}_1)}$ is given by $ \langle v^{(n)}_1, (W_n-z)^{-1} v^{(n)}_1 \rangle$ which is sometimes called a generalized entry of the resolvent and has already been studied in the literature. The most recent result is the local law recently obtained by Knowles and Yin [16]. It consists in a uniform estimation of $\langle v,(W_n-z)^{-1} w \rangle$ for any vectors $v$ and $w$ and for any complex $z$ in a domain of the upper half plane that is allowed to approach the real axis as n tends to infinity. Since it is one ingredient of the proof of the forthcoming Theorem \ref{theo: square proj Wigner general case}, we provide a precise statement. 

Let us first introduce some notations. For all $n \geq 1$, writing $z=E+i\eta$, we define:
\begin{equation}\label{eq:Definition1}
\begin{cases}
G_n(z) = \left( W_n - z  \right)^{-1}, \\
\Pi_n(z) = \left( A_n - z - s_{ \mu_{sc} \boxplus \mu_{A} } (z)  \right)^{-1}, \\
\psi_n (z) = \sqrt{ \frac{\Im \left( s_{ \mu_{sc} \boxplus \mu_A } (z) \right) }{n \eta} } + \frac{1}{n \eta}.
\end{cases}
\end{equation}
For all $x \in \mathbf{R}$, $c>0$ and $\tau>0$, we also consider:
\begin{equation}\label{eq:Specdomain1}
\mathcal{D}_n^{(\tau)}(x,c) := \left\{ z \in \mathbf{C}, \, x-c \leq E \leq x+c, \, n^{-1 + \tau} \leq \eta \leq \tau^{-1}    \right\}. 
\end{equation}

The local law we will use consists in a uniform control between the generalized entries of $G_n$ and $\Pi_n$ in the spectral domain $\mathcal{D}_n^{(\tau)}(x,c)$ whenever the density of $\mu_{sc} \boxplus \mu_A$ is bounded away from $0$ on the interval $[x-c,x+c]$.

\begin{Theo}{ {\bf \cite[Theorem 12.2]{MR3704770} }}  \label{theo: local law Wigner}
Let $x \in \mathbf{R}$ and $c>0$. Suppose that:
\begin{equation}\label{eq:positivedens1}
\inf\limits_{ t \in [x-c,x+c] } \frac{\mathrm{d}(\mu_{sc} \boxplus \mu_A)(t)}{\mathrm{d}t} > 0. 
\end{equation}
Then, for any $\tau >0$, uniformly in all vectors $v,w$ and uniformly in $z \in \mathcal{D}_n^{(\tau)}(x,c)$, for all $\varepsilon >0$, there exists $D >0$ such that 
\begin{equation}  \label{eq:theolocWign} 
\PP\left( \left| \langle v, G_n(z)w \rangle - \langle v, \Pi_n(z) w \rangle \right| \geq n^\varepsilon       \psi_n (z) |  | v |  | \, |  | w | | \right) \leq \frac{1}{n^D}.  
\end{equation}
\end{Theo}
Theorem \ref{theo: local law Wigner} is a direct consequence of the work of Knowles and Yin \cite{MR3704770}. More precisely, Theorem 12.2 of \cite{MR3704770} provides a local law of the form \eqref{eq:theolocWign} uniformly in a spectral domain $\mathbf{S}_n \subset \mathbf{C}_+$ provided that an entrywise local law (meaning that $v$ and $w$ are vectors of the canonical basis in \eqref{eq:theolocWign}) has been proved in the particular case where $A_n$ is diagonal, uniformly in the same spectral domain $\mathbf{S}_n$. Such a result has indeed been established by Lee, Schnelli, Stetler and Yau in \cite[Theorem 3.3]{MR3502606}.

Let us comment on hypothesis \eqref{eq:positivedens1}. Together with the assumption that the eigenvalues $\gamma_i^{(n)}$'s remain bounded, it implies that there exists a constant $C>0$ such that, uniformly in $z \in \mathcal{D}_n^{(\tau)}(x,c)$, for all $1 \leq i \leq n$,
\begin{equation}\label{eq:stabilityeq1}
\frac{1}{C} \leq \left| \gamma_i^{(n)} - z - s_{\mu_{sc} \boxplus \mu_A}(z)   \right| \leq C.
\end{equation} 
Equation \eqref{eq:stabilityeq1} is sometimes referred to as the {\it stability assumption}. Going back to the proofs of the local laws, it can be checked that, whenever \eqref{eq:stabilityeq1} is satisfied on a spectral domain $\mathbf{S}_n$, then \eqref{eq:theolocWign} can be proved on $\mathbf{S}_n$. In particular, since \eqref{eq:stabilityeq1} can hold without assumption \eqref{eq:positivedens1}, the local law \eqref{eq:theolocWign} is usually proved on larger domains that $\mathcal{D}_n^{(\tau)}(x,c)$. We choose to state it on $\mathcal{D}_n^{(\tau)}(x,c)$ because we only need this weaker version in the proof of Theorem \ref{theo: square proj Wigner general case}.

\bigskip

The non-local counterpart of Theorem \ref{theo: local law Wigner} is the pointwise convergence of $\langle v, (W_n - z)^{-1} w \rangle$ in the domain $\Im(z) >0$. Taking $v=w=v_1^{(n)}$ yields the following Corollary. 

\begin{corollary}\label{coro: Wigner case, as convergence}
The spectral measure $\mu_{(W_n,v^{(n)}_1)}$ converges in probability towards a deterministic probability measure $\mu_{sc,A,\theta}$ whose Stieltjes transform is given by
\begin{equation} 
s_{\mu_{sc,A,\theta}}(z) = \frac{1}{\theta - s_{\mu_{sc} \boxplus \mu_A }(z) - z} . 
\label{eq: Stieltjes transform eq Wigner case}
\end{equation}
\end{corollary}

\begin{remark}\label{remark: AB}
Equation \eqref{eq: Stieltjes transform eq Wigner case} could allow to retrieve the limit of $\frac{1}{n} \E \left[ \mathrm{Tr} \left( (W_n - z)^{-1} g(A_n)   \right) \right]$, obtained in \cite[Lemme 5.1]{MR3256861} by Allez and Bouchaud, for any measurable function $g$. Indeed, this quantity can be rewritten 
\[ \frac{1}{n} \E \left[ \sum_{i,j=1}^n \frac{ | \langle \phi_i^{(n)}, v_j^{(n)} \rangle |^2 }{ \lambda_i^{(n)} - z} g\left( \gamma_j^{(n)}  \right)  \right], \]
which is nothing but the average of the Stieltjes transforms of the pushforward of the spectral measures of $W_n$ by $g$. In particular, it converges to a non-degenerate limit only when the support of $g$ is contained into a {\it macroscopic} part of $\mathrm{Supp}(\mu_A)$, due to the renormalization by $n$. When $g$ is non-null on a {\it microscopic} part of $\mathrm{Supp}(\mu_A)$, the study of the spectral measures allows to obtain the limit of $\E \left[ \mathrm{Tr} \left( (W_n - z)^{-1} g(A_n)   \right) \right]$ whereas $\frac{1}{n} \E \left[ \mathrm{Tr} \left( (W_n - z)^{-1} g(A_n)   \right) \right]$ brings no information as it converges to zero.
\end{remark}

Note that such a macroscopic result can be obtained using more simple arguments than the local law of Theorem \ref{theo: local law Wigner}. See for example \cite[Proposition $6.2$]{MR3090543} in the case where $v$ and $w$ are vectors of the canonical basis.

We provide two applications of the asymptotic behavior of the spectral measure of $W_n$ in the direction of $v_1^{(n)}$. 

The first one is concerned with outliers and the projection in the direction of the spike of their associated eigenvectors and relies on the following observation: unlike the empirical spectral measure which contains information on outliers only at the order $1/n$, the spectral measure in the direction of the spike already contains it at a {\it macroscopic} order. For all $x \in \mathbf{R} \setminus \mathrm{Supp}( \mu_{sc} \boxplus \mu_A)$, let us introduce 
\[ w(x) := x + s_{\mu_{sc,A,\theta}}(x). \]
If there exists $x$ such that $w(x) = \theta$, it is easy to deduce the existence of outliers for $W_n$ as explained in the following Corollary. Although it is already-known in random matrix theory, our approach is new and relatively simple.
\begin{corollary}\label{coro: Wigner case}
Suppose that there exists $x_{\theta} \notin \mathrm{Supp}(\mu_{sc} \boxplus \mu_A)$ such that $w(x_\theta) = \theta$. Then, $x_{\theta}$ is an outlier of $W_n$. More precisely, set $\delta >0$ such that $[x_\theta - \delta, x_\theta + \delta ] \cap \mathrm{Supp}(\mu_{sc} \boxplus \mu_A) = \emptyset$ and define $k_n$ to be the number of eigenvalues of $W_n$ inside $[x_\theta - \delta, x_\theta + \delta ]$. There exists $1 \leq i_n \leq n$ such that these eigenvalues satisfy
\[  x_\theta + \delta \geq \lambda_{i_n + 1}^{(n)} \geq \lambda_{i_n + 2}^{(n)}  \geq  \cdots \geq \lambda_{i_n + k_n}^{(n)} \geq x_\theta - \delta.  \]
Then, $k_n \geq 1$ for $n$ sufficiently large and:
\begin{enumerate}
\item Both $\lambda_{i_n + 1}^{(n)}$ and $\lambda_{i_n + k_n}^{(n)}$ converge in probability towards $x_\theta$;
\item $ \sum\limits_{p=1}^{k_n} | \langle \phi_{i_n + p}^{(n)}, v_1^{(n)} \rangle |^2$ converges in probability towards $\frac{1}{w'(x_\theta)}$.
\end{enumerate}
\end{corollary}
\begin{proof}
Let $x_{\theta} \notin \mathrm{Supp}(\mu_{sc} \boxplus \mu_A)$ be such that $w(x_{\theta})=\theta$. The value of $\mu_{sc,A,\theta}(\{x_\theta \})$ is given by the residue of $s_{\mu_{sc,A,\theta}}$ at $x_\theta$:
\[ (x_{\theta}-z)s_{\mu_{sc,A,\theta}}(z) = \frac{x_{\theta}-z}{w(x_{\theta})-w(z)} \underset{ z \rightarrow x^+ }{\longrightarrow} \frac{1}{w'(x_{\theta})} > 0.  \]
Since $\mu_{(W_n,v_1^{(n)})}$ converges towards $\mu_{sc,A,\theta}$ by Proposition \ref{coro: Wigner case, as convergence}, the Corollary is proved.
\end{proof} 
In particular, when $W_n$ is known to be a rank-one perturbation of a matrix $W_n'$ whose empirical spectral measure converges towards $\mu_{sc} \boxplus \mu_A$ and which contains no outlier, the interlacing property implies that $W_n$ has a unique outlier, namely that $k_n=1$ in Corollary \ref{coro: Wigner case}. Therefore, in that case, the unique outlier converges towards $x_\theta$ and the square projection in the direction of the spike of its associated eigenvector converges towards $1/w'(x_\theta)$.

Before stating our the second result, which represents the main novelty of this paper, and is also an illustration of the use of the spectral measure, we need the following observation:
\begin{prop}\label{prop: density Wigner}
 $\mu_{sc,A,\theta}$ is absolutely continuous with respect to the Lebesgue measure on $\mathrm{Supp}( \mu_{sc} \boxplus \mu_A )$.
\end{prop}
\begin{proof}
In \cite{MR1488333}, Biane proved that $\mu_{sc} \boxplus \mu_A$ is absolutely continuous with respect to the Lebesgue measure. Therefore, the inverse formula 
\begin{equation}  \frac{\mathrm{d\mu_{sc,A,\theta}}(x)}{\mathrm{d }x} =  \frac{1}{\pi} \lim\limits_{ t \rightarrow 0^+} \Im(s_{\mu_{sc,A,\theta}}(x+it) )  
\label{eq: Stieltjes inversion Wign}
\end{equation}
and Equation \eqref{eq: Stieltjes transform eq Wigner case} imply that $\mu_{sc,A,\theta}$ is also absolutely continuous with respect to the Lebesgue measure at any $x \in \mathrm{Supp}( \mu_{sc} \boxplus \mu_A )$.
\end{proof} 

We will denote $f_{sc,A}$ and $f_{sc,A,\theta}$ the respective densities of $\mu_{sc} \boxplus \mu_A$ and $\mu_{sc,A,\theta}$ on $\mathrm{Supp}( \mu_{sc} \boxplus \mu_A )$ (these are well-defined quantities by Proposition \ref{prop: density Wigner}). It turns out that the averaged square-projections of the non-outlier eigenvectors associated to eigenvalues in the vicinity of $x \in \mathrm{Supp}( \mu_{sc} \boxplus \mu_A) $ converges towards the ratio of these two densities.
\begin{theorem}\label{theo: square proj Wigner general case}
Let $x \in \mathrm{Supp}(\mu_{sc} \boxplus \mu_A) $ be such that $f_{sc,A}(x) > 0$. Let $\varepsilon_n$ be a sequence that satisfies $n^\delta / \sqrt{n} <\! \!< \varepsilon_n <\! \!< 1$ for some $0<\delta<1/2$. Then, for every $t>0$, if $\mathcal{I}^{(n)}_{\varepsilon_n}(x) = \left\{ 1 \leq i \leq n: \, | \lambda_i^{(n)} - x | \leq \varepsilon_n   \right\}$:
\[  \PP \left( \left|  \frac{n}{|\mathcal{I}^{(n)}_{\varepsilon_n}(x)|} \sum\limits_{i \in \mathcal{I}^{(n)}_{\varepsilon_n}(x)} \! \! \left| \langle \phi^{(n)}_i,v^{(n)}_1 \rangle \right|^2 - \frac{f_{sc,A,\theta}(x)}{f_{sc,A}(x)} \right| > t    \right)  \underset{n \rightarrow + \infty}{ \longrightarrow} 0.  \]
\end{theorem}

By taking $g$ the indicator of an interval contained in $\mathrm{Supp}(\mu_A)$ into the statistic introduced in Remark \ref{remark: AB}, Allez and Bouchaud \cite{MR3256861} obtained the asymptotic behavior of the overlaps $|\langle \phi_i^{(n)}, v_j^{(n)} \rangle |^2$ after taking average over eigenvectors $\phi_i^{(n)}$'s (resp. $v_j^{(n)}$'s) with associated eigenvalues $\lambda_i^{(n)}$'s belonging to a {\it macroscopic} proportion of $\mathrm{Supp}(\mu_{sc} \boxplus \mathrm{Supp}(\mu_A))$ (resp. $\mu_A$). When $\theta \in \mathrm{Supp}( \mu_A)$, Theorem \ref{theo: square proj Wigner general case} confirms their result at a {\it microscopic} scale. Indeed, denoting respectively $a$ and $b$ the real and imaginary parts of $\frac{1}{\pi} \lim_{t \rightarrow 0^+}    s_{\mu_{sc} \boxplus \mu_A} (x+it)$, one can rewrite, using the inverse formula \eqref{eq: Stieltjes inversion Wign}: 
\[ \frac{f_{sc,A,\theta}(x)}{f_{sc,A}(x)} = \frac{1}{( \theta -x - a)^2 + b^2}.  \]
When $\theta \notin \mathrm{Supp}(\mu_A)$, the approach of \cite{MR3256861} provides no information on the overlap because it only gives access to $n^{-1}s_{\mu_{(W_n,v_1^{(n)})}}(z)$ which converges to zero, whereas the spectral measure approach still works.

\subsection{The rank-one perturbation}\label{subsec: Wigner explicit}
In the special case where $\gamma^{(n)}_2 = \cdots = \gamma^{(n)}_n = 0$ for all $n\geq 1$, $W_n$ is a rank-one perturbation of a classical Wigner matrix. The limiting spectrum of the perturbation is $\mu_A=\delta_0$ and almost surely, $\mu_{W_n}$ weakly converges towards the semicircle distribution. In this setting, we provide explicit computations. Proposition \ref{coro: Wigner case, as convergence} has now the more explicit formulation:
\begin{prop}\label{prop: Wigner case explicit comput of limiting measure}
In probability, $\mu_{(W_n,v^{(n)}_1)}$ converges towards:
\[ \mu_{sc,\theta}(\mathrm{d}x) := \frac{\sqrt{4-x^2}}{2 \pi (\theta^2 + 1 - \theta x)} \mathbf{1}_{ |x| \leq 2} \mathrm{d}x + \mathbf{1}_{|\theta|>1} \left( 1 - \frac{1}{\theta^2} \right) \delta_{\theta + \frac{1}{\theta}} (\mathrm{d}x).    \]
\end{prop}
Remark that $W_n$ is a rank-one perturbation of $n^{-1/2} X_n$. Therefore, since $\lambda_1(n^{-1/2}X_n) \rightarrow 2$ and $\lambda_n(n^{-1/2}X_n) \rightarrow -2$ in probability (see \cite{MR637828}), $W_n$ has a single outlier whose location is given by the atom of $\mu_{sc,\theta}$ and whose associated eigenvector has a square projection in the direction of the spike given by the mass of this atom.
\begin{corollary} The following holds:
\begin{enumerate}
\item If $\theta > 1 $, then, in probability, $\lambda_1(W_n)\underset{n \rightarrow + \infty}{\longrightarrow} \theta + \frac{1}{\theta} >\phantom{-} 2$ and $|\langle \phi^{(n)}_1,v^{(n)}_1 \rangle| \underset{n \rightarrow + \infty}{\longrightarrow} \sqrt{ 1 - \frac{1}{\theta^2} }$.
\item If $\theta < -1 $, then, in probability, $\lambda_n(W_n) \underset{n \rightarrow + \infty}{\longrightarrow} \theta + \frac{1}{\theta} < -2$ and $|\langle \phi^{(n)}_n,v^{(n)}_1 \rangle| \underset{n \rightarrow + \infty}{\longrightarrow} \sqrt{ 1 - \frac{1}{\theta^2} } $.
\end{enumerate}
\end{corollary}
Finally, the averaged square-projections have also an explicit form, which is just the inverse of a linear function in that case:
\begin{theorem}\label{theo: averaged explicit Wigner}
Let $x\in (-2,2)$. Let $\varepsilon_n$ be a sequence that satisfies $n^\delta / \sqrt{n} <\! \!< \varepsilon_n < \! \! < 1$ for some $\delta >0$. Then, for every $t>0$,
\[  \PP \left( \left|  \frac{n}{|\mathcal{I}^{(n)}_{\varepsilon_n}(x)|} \sum\limits_{i \in \mathcal{I}^{(n)}_{\varepsilon_n}(x)} \left| \langle \phi^{(n)}_i,v^{(n)}_1 \rangle \right|^2 - \frac{1}{\theta^2-\theta x + 1} \right| > t   \right)  \underset{n \rightarrow + \infty}{ \longrightarrow} 0.  \]
\end{theorem}

\begin{figure}[ht!]
\centering
\begin{subfigure}{.5\textwidth}
  \centering
  \includegraphics[scale=0.4]{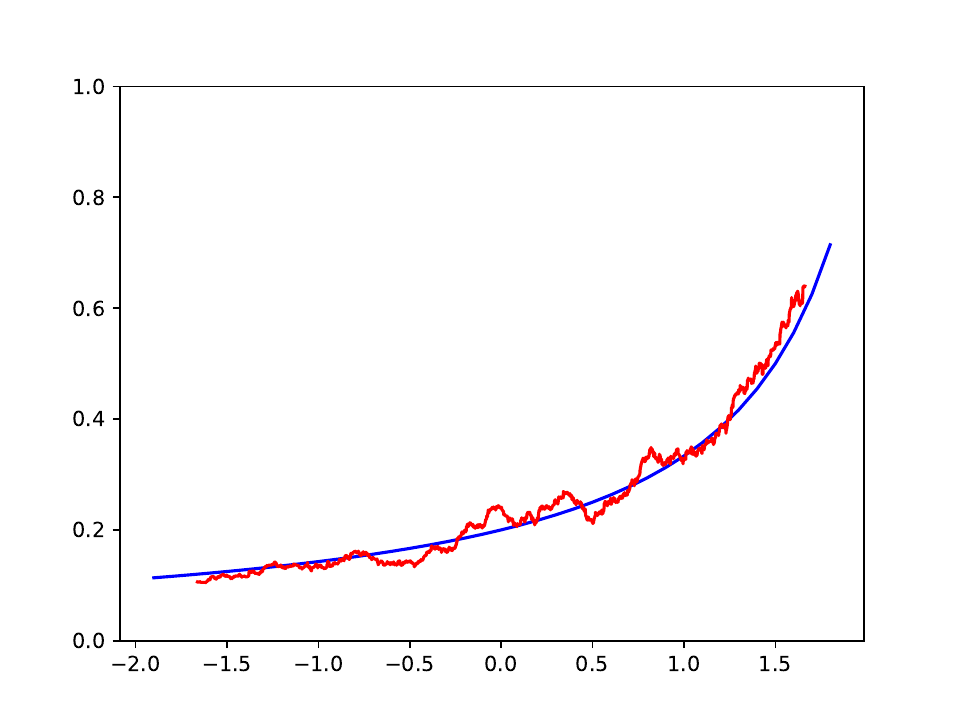}  
  \caption{}
  \label{fig:sfig11 Wigner}
\end{subfigure}%
\begin{subfigure}{.5\textwidth}
  \centering
  \includegraphics[scale=0.4]{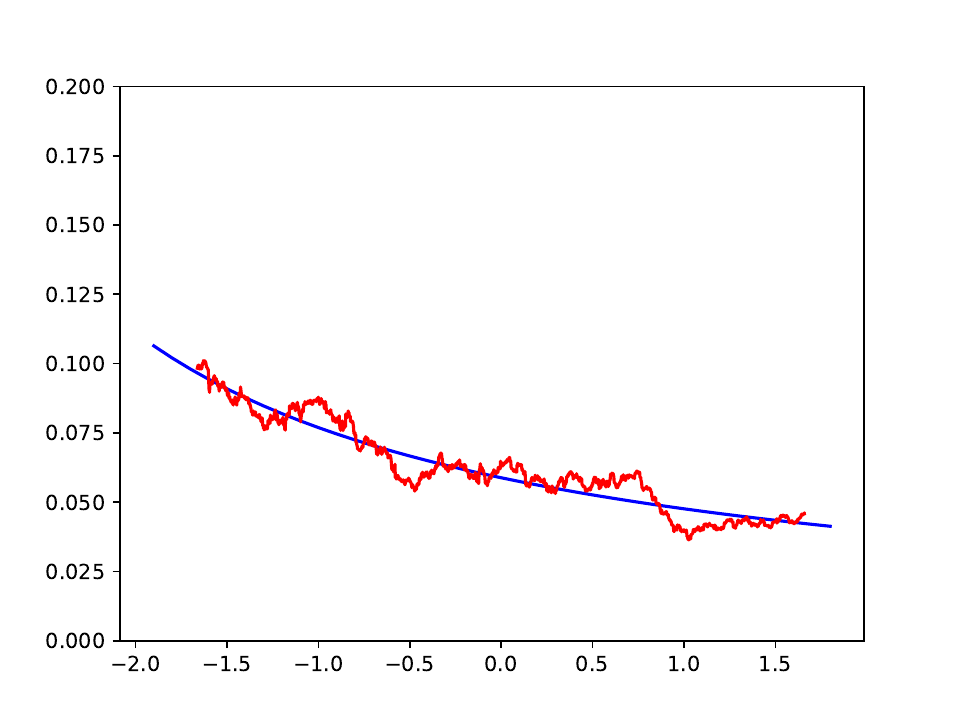} 
  \caption{}
  \label{fig:sfig21 Wigner}
\end{subfigure}
\caption{{\footnotesize In red: simulations of the average squared-projections around all locations $x\in(-2,2)$ where we took average over intervals of typical size $n^{0,1}/\sqrt{n}$ for a single matrix $W_n = n^{-1/2}X_n$, where $X_n$ has gaussian entries and is of size $3000 \times 3000$. In case (a) $\theta=2$ and in case (b) $\theta=-4$. In blue: theoretical predictions.}}
\label{fig: averaged square proj1 Wigner}
\end{figure}

\section{Multiplicative perturbation of a Wishart matrix}\label{sec: Wishart general}
\subsection{The general framework}\label{subsec: Wishart general framework}
Let $m=m(n)$ be a sequence of integers such that $m/n \rightarrow \alpha >0$ as $n \rightarrow + \infty$. For all $n \geq 0$, we consider the following random rectangular matrix:
\[ X_n := \left[
\begin{array}{cccc}
   X^{(n)}_{11} & X^{(n)}_{12} & \cdots & X^{(n)}_{1m} \\
   X^{(n)}_{21} & X^{(n)}_{22} & \cdots & X^{(n)}_{2m} \\
   \vdots       & \vdots       & \ddots & \vdots  \\
   X^{(n)}_{n1} & X^{(n)}_{n2} & \cdots & X^{(n)}_{nm}
\end{array}
\right].  \]
Let also $\Sigma_n$ be a general covariance matrix of size $n \times n$, with eigenvalues given by $\gamma^{(n)}_1 =\theta, \gamma^{(n)}_2, \ldots, \gamma^{(n)}_n$ and associated eigenvectors $v^{(n)}_1, \ldots, v^{(n)}_n$. We suppose that there exists a probability measure $\mu_{\Sigma}$ such that
\[ \frac{1}{n} \sum\limits_{i=1}^n \delta_{\gamma^{(n)}_i} \underset{n \rightarrow + \infty}{ \longrightarrow} \mu_{\Sigma} \] 
in the sense of weak convergence. Moreover, let us assume that there exists $\Gamma>0$ such that for all $ \geq 1$, $\sup_{1 \leq i \leq n} |\gamma_i^{(n)}| \leq \Gamma$.
We study the following multiplicative perturbation model:
\[  S_n := \frac{1}{n} \Sigma_n^{1/2} X_n X^T_n \Sigma^{1/2}_n.   \]
The matrix $S_n$ can be considered as the sampled covariance matrix of $m$ i.i.d. vectors in $\mathbf{R}^n$ having covariance matrix $\Sigma_n$. Let $\lambda^{(n)}_1 \geq \cdots \geq \lambda^{(n)}_n$ be the eigenvalues of $S_n$ and $\phi^{(n)}_1, \ldots, \phi^{(n)}_n$ the associated normalized eigenvectors.

The empirical spectral distribution of $S_n$ converges towards the free product $\mu_{\mathrm{MP},\alpha} \boxtimes \mu_{\Sigma}$ whose Stieltjes transform is characterized by:
\[  s_{\mu_{\mathrm{MP},\alpha} \boxtimes \mu_{\Sigma}}(z)  =  \int_{\mathbf{R}} \frac{\mathrm{d}\mu_{\Sigma}(t)}{t(\alpha-1-z s_{\mu_{\mathrm{MP},\alpha} \boxtimes \mu_{\Sigma}}(z) ) - z} .  \]
This is a consequence of the work of Silverstein \cite{MR1370408}. A later work of Cho\"i and Silverstein \cite{MR1345541} proved that $\mu_{\mathrm{MP},\alpha} \boxtimes \mu_{\Sigma}$ is absolutely continuous with respect to the Lebesgue measure at $x$, for any $x>0$.

The study of $S_n$ is intimately linked to the study of $\underline{S}_n:=\frac{1}{n}X_n^T \Sigma_n X_n$. Indeed, the spectra of $S_n$ and $\underline{S}_n$ only differ by the number of zero eigenvalues. More precisely, denoting $\mu_{S_n}$ and $\mu_{\underline{S}_n}$ the respective empirical spectral measures of $S_n$ and $\underline{S}_n$:
\[ \mu_{\underline{S}_n} = \frac{1}{n}\sum\limits_{ \lambda \in \mathrm{Spec}(S_n) } \delta_\lambda + \frac{m-n}{n}\delta_0. \]
In term of the Stieltjes transforms, this relation translates into $s_{\mu_{\underline{S}_n}}(z) = s_{ \mu_{S_n} }(z) - (\alpha -1)/z$. Therefore, when $n$ tends to infinity, the empirical spectral measure of $\underline{S}_n$ converges towards a deterministic measure whose Stieltjes transform, denoted $\underline{s}$, is given by:
\begin{equation}\label{eq:defnunderline}
\underline{s}(z) = s_{\mu_{\mathrm{MP},\alpha} \boxtimes \mu_{\Sigma}}(z) - \frac{\alpha-1}{z}.
\end{equation}

\begin{remark}{{\bf (Scaling conventions).}}
In the literature of random covariance matrices, many authors are using different scaling than this paper. Namely, they study Wishart matrices of the form $B_N = \frac{1}{N} Y_N Y_N^T$ where $Y_N$ is of size $p \times N$ where $p/N \rightarrow y >0$ as $N$ tends to infinity. In this context, the scaling factor is the dimension of the rows vectors of $Y_N$ whereas in our case, we scale by the dimension of the columns vectors. In order to facilitate the comparison between the two models, let us describe their differences. First, note the following correspondence between the parameters: $n=p$, $m=N$ and $\alpha = 1/y$. Writing $B_N = (p/N) \frac{1}{p}Y_N Y_N^T$, it is easy to see that the empirical spectral measures of the two models satisfy the following equality in law:
\[ \mu_{B_N} = \Lambda_{p/N} \left( \mu_{S_n} \right) = \Lambda_{n/m}\left( \mu_{S_n} \right),  \]
where $\Lambda_\xi (\cdot)$ denotes the pushforward by the dilatation $x \mapsto \xi x$. Hence, $\mu_{B_N}$ converges towards the pushforward of $\mu_{\mathrm{MP},\alpha} \boxtimes \mu_{\Sigma}$ by $\Lambda_{1/\alpha}$. In the particular case where $\Sigma_n$ is the identity, the limiting measure is given by the pushforward of $\mu_{\mathrm{MP},\alpha}$ which, after computations, is given by
\[ \frac{\sqrt{((1+\sqrt{y})^2-x)(x-(1-\sqrt{y})^2)}}{2 \pi y x} \mathbf{1}_{x \in ( (1-\sqrt{y})^2, (1+\sqrt{y})^2 ) } \mathrm{d}x + \left(1 - \frac{1}{y} \right) \delta_0(\mathrm{d}x). \]
\end{remark}

We think of $\theta$ as a spike, that is an atypical eigenvalue compared to the sequence $\gamma^{(n)}_i$, $2 \leq i \leq n$. We are interested in its influence on the apparition of outliers for $S_n$. To that purpose, we study the spectral measure in the direction of the spike:
\[  \mu_{(S_n,v_1^{(n)})} := \sum\limits_{i=1}^n | \langle \phi^{(n)}_i, v^{(n)}_1   \rangle |^2 \delta_{\lambda^{(n)}_i}. \]

The Stieltjes transform of $\mu_{(S_n,v^{(n)}_1)}$ is given by $ \langle v^{(n)}_1, (S_n-z)^{-1} v^{(n)}_1 \rangle$ and has already been studied in the literature. As in the Wigner case, the most recent result is the local law obtained by Knowles and Yin \cite{MR3704770}. It consists in a uniform estimation of $ \langle v , (S_n-z)^{-1} w \rangle$ for any vectors $v$ and $w$ and for any complex $z$ in a domain of the upper half plane that is allowed to approach the real axis as $n$ tends to infinity. Since it is an ingredient of the proof of Theorem \ref{theo: square proj Wishart general case}, we provide a precise statement. 

Let us first introduce some notations. For all $n \geq 1$, writing $z=E+i\eta$, we define:
\begin{equation}\label{eq:Definition2}
\begin{cases}
G_n(z) = \left( S_n - z  \right)^{-1}, \\
\Pi_n(z):= - (z(1 + \underline{s}(z)\Sigma_n ))^{-1}, \\
\psi_n (z) = \sqrt{ \frac{\Im \left( s_{ \mu_{\mathrm{MP},\alpha} \boxtimes \mu_\Sigma } (z) \right) }{n \eta} } + \frac{1}{n \eta},
\end{cases}
\end{equation}
where we recall that $\underline{s}$ is defined in Equation \eqref{eq:defnunderline}. For all $x \in \mathbf{R}$, $c>0$ and $\tau>0$, we also consider:
\begin{equation}\label{eq:Specdomain2}
\mathcal{D}_n^{(\tau)}(x,c) := \left\{ z \in \mathbf{C}, \, x-c \leq E \leq x+c, \, n^{-1 + \tau} \leq \eta \leq \tau^{-1}    \right\}. 
\end{equation}

The local law we will use consists in a uniform control between the generalized entries of $G_n$ and $\Pi_n$ in the spectral domain $\mathcal{D}_n^{(\tau)}(x,c)$ whenever the density of $\mu_{\mathrm{MP},\alpha} \boxtimes \mu_\Sigma$ is bounded away from $0$ on the interval $[x-c,x+c]$.

\begin{Theo}{\bf \cite[Corollary 3.9]{MR3704770}}\label{theo: local law Wishart}
Let $x >0$ and $c>0$. Suppose that:
\begin{equation}\label{eq:positivedens2}
\inf\limits_{ t \in [x-c,x+c] } \frac{\mathrm{d}(\mu_{\mathrm{MP},\alpha} \boxtimes \mu_\Sigma)(t)}{\mathrm{d}t} > 0. 
\end{equation}
Then, for any $\tau >0$, uniformly in all vectors $v,w$ and uniformly for in any $z \in \mathcal{D}_n^{(\tau)}$, for all $\varepsilon >0$, there exists $D >0$ such that 
\begin{equation}   
\PP\left( \left| \langle v, \Sigma_n^{-1/2}  \left(  G_n(z) - \Pi_n(z) \right) \Sigma_n^{-1/2}  w \rangle  \right| \geq n^\varepsilon       \psi_n (z) |  | v |  | \, |  | w | | \right) \leq \frac{1}{n^D}.  
\label{eq: estimate local law Wishart}
\end{equation}
\end{Theo}

Theorem \ref{theo: local law Wishart} is a direct consequence of the work of Knowles and Yin \cite{MR3704770}. More precisely, they first prove \cite[Theorem 3.22]{MR3704770} that \eqref{eq: estimate local law Wishart} holds in the particular case where $v$ and $w$ are vectors of the canonical basis, and when $\Sigma_n$ is diagonal. Then, using a comparison argument, they extend this result to the general setting (see \cite[Theorem 3.21]{MR3704770}).

It is possible to comment hypothesis \eqref{eq:positivedens2} as in the Wigner case (see the paragraphs below Theorem \ref{theo: local law Wigner}). Together with the assumption that the eigenvalues $\gamma_i^{(n)}$'s remain bounded, it implies that there exists a constant $C>0$ such that, uniformly in $z \in \mathcal{D}_n^{(\tau)}(x,c)$, for all $1 \leq i \leq n$,
\begin{equation}\label{eq:stabilityeq2}
\frac{1}{C} \leq \left| 1 +  \gamma_i^{(n)}  \underline{s}(z)   \right| \leq C.
\end{equation} 
Equation \eqref{eq:stabilityeq2} is sometimes referred to as the {\it stability assumption}. Going back to the proofs of the local laws, it can be checked that, whenever \eqref{eq:stabilityeq2} is satisfied on a spectral domain $\mathbf{S}_n$, then \eqref{eq: estimate local law Wishart} can be proved on $\mathbf{S}_n$. In particular, since \eqref{eq:stabilityeq2} can hold without assumption \eqref{eq:positivedens2}, the local law \eqref{eq: estimate local law Wishart} is usually proved on larger domains that $\mathcal{D}_n^{(\tau)}(x,c)$. We choose to state it on $\mathcal{D}_n^{(\tau)}(x,c)$ because we only need this weaker version in the proof of Theorem \ref{theo: square proj Wishart general case}.

\bigskip

The non-local counterpart of Theorem \ref{theo: local law Wigner} is the pointwise convergence of $\langle v, (S_n - z)^{-1} w \rangle$ in the domain $\Im(z) >0$. Taking $v=w=v_1^{(n)}$ yields the following Corollary.

\begin{corollary}\label{coro: Wishart general case as convergence}
In probability, $\mu_{(S_n,v^{(n)}_1)}$ weakly converges towards a probability measure $\mu_{\alpha,\Sigma,\theta}$ whose Stieltjes transform is given by:
\begin{equation} s_{\mu_{\alpha,\Sigma,\theta}}(z) = \frac{1}{\theta(\alpha-1) - z \theta s_{\mu_{\mathrm{MP},\alpha} \boxtimes \mu_{\Sigma}}(z) - z}. 
\label{eq: Stieltjes transform Wishart case}
\end{equation}
\end{corollary}

\begin{remark}\label{remark: LD}
Equation \eqref{eq: Stieltjes transform Wishart case} could allow to retrieve the limit of $\frac{1}{n} \mathrm{Tr} \left( (S_n - z)^{-1} g(\Sigma_n)   \right)$, obtained in \cite[Theorem 2]{MR2834718} by Ledoit and P\'ech\'e, for any measurable function $g$. Indeed, this quantity can be rewritten 
\[ \frac{1}{n}\sum_{i,j=1}^n \frac{ | \langle \phi_i^{(n)}, v_j^{(n)} \rangle |^2 }{ \lambda_i^{(n)} - z} g\left( \gamma_j^{(n)}  \right), \]
which is nothing but the average of the Stieltjes transform of the pushforward of the spectral measures of $S_n$ by $g$. In particular, it converges to a non-degenerate limit only when the support of $g$ is contained into a {\it macroscopic} part of $\mathrm{Supp}(\mu_\Sigma)$, due to the renormalization by $n$. When $g$ is non-null on a {\it microscopic} part of $\mathrm{Supp}(\mu_\Sigma)$, the study of the spectral measures yields the limit of $\mathrm{Tr} \left( (S_n - z)^{-1} g(\Sigma_n)   \right)$ whereas $\frac{1}{n} \mathrm{Tr} \left( (S_n - z)^{-1} g(\Sigma_n)   \right)$ brings no information as it converges to zero.
\end{remark}

Of course, such a macroscopic result can be obtained using more simple arguments than the local law of Theorem \ref{theo: local law Wishart}. We refer for example to \cite[Proposition $6.2$]{MR3090543}.

In what follows, we provide two applications of the asymptotic behavior of the spectral measure of $S_n$ in the direction of $v_1^{(n)}$. 

The first one recovers a classical result on outliers of $S_n$ and the projection in the direction of the spike of their associated eigenvectors. The proof we propose is simple and based on the following observation: unlike the empirical spectral measure which contains information on outliers only at the order $1/n$, the spectral measure in the direction of the spike already contains it at a {\it macroscopic} order. Before stating our result, let us introduce, for all $x>0$ such that $x^{-1} \in \mathbf{R} \setminus \mathrm{Supp}(\mu_{\mathrm{MP},\alpha} \boxtimes \mu_\Sigma)$:
\[ F(x) = \alpha x - x - s_{ \mu_{\mathrm{MP},\alpha} \boxtimes \mu_\Sigma}(1/x). \] 
If there exists $x>0$ such that $1/F(1/x) = \theta$, we easily obtain the existence of an outlier as explained in the following Corollary.
\begin{corollary}\label{coro: BBP Wishart case}
Suppose that there exists $x_{\alpha,\theta} \notin \mathrm{Supp}(\mu_{\mathrm{MP},\alpha} \boxtimes \mu_{\Sigma})$ such that $1/F(1/x_{\alpha,\theta})= \theta$. Then, $x_{\alpha,\theta}$ is an outlier of $S_n$. More precisely, set $\delta >0$ such that $[x_{\alpha,\theta} - \delta, x_{\alpha,\theta} + \delta ] \cap \mathrm{Supp}(\mu_{sc} \boxplus \mu_A) = \emptyset$ and define $k_n$ to be the number of eigenvalues of $S_n$ inside $[x_{\alpha,\theta} - \delta, x_{\alpha,\theta} + \delta ]$. There exists $1 \leq i_n \leq n$ such that these eigenvalues satisfy
\[  x_{\alpha,\theta} + \delta \geq \lambda_{i_n + 1}^{(n)} \geq \lambda_{i_n + 2}^{(n)}  \geq  \cdots \geq \lambda_{i_n + k_n}^{(n)} \geq x_{\alpha,\theta} - \delta.  \]
Then, $k_n \geq 1$ for $n$ sufficiently large and:
\begin{enumerate}
\item Both $\lambda_{i_n + 1}^{(n)}$ and $\lambda_{i_n + k_n}^{(n)}$ converge in probability towards $x_{\alpha,\theta}$;
\item $ \sum\limits_{p=1}^{k_n} | \langle \phi_{i_n + p}^{(n)}, v_1^{(n)} \rangle |^2$ converges in probability towards $\frac{x_{\alpha,\theta}F(1/x_{\alpha,\theta})}{ F'(1/x_{\alpha,\theta})}$.
\end{enumerate}
\end{corollary}
\begin{proof}
Let $x_{\alpha,\theta}>0$ be such that $x_{\alpha,\theta} \notin \mathrm{Supp}(\mu_{\mathrm{MP},\alpha} \boxtimes \mu_{\Sigma})$ and $1/F(1/x_{\alpha,\theta})=\theta$. The value of $\mu_{\alpha,\Sigma,\theta}( \{x_{\alpha,\theta} \})$ is given by the residue of $s_{\mu_{\alpha,\Sigma,\theta}}$ at $x_{\alpha,\theta}$:
\[ (x_{\alpha,\theta}-z)s_{\mu_{\alpha,\Sigma,\theta}}(z) = \frac{(x_{\alpha,\theta}-z) F(1/x_{\alpha,\theta}) }{z (F(1/z) - F(1/x_{\alpha,\theta})) } \underset{ z \rightarrow x_{\alpha,\theta}^+ }{\longrightarrow} \frac{x_{\alpha,\theta}F(1/x_{\alpha,\theta})}{ F'(1/x_{\alpha,\theta})} > 0.  \]
Since $\mu_{(S_n,v_1^{(n)})}$ converges towards $\mu_{\alpha, \Sigma, \theta}$, the Corollary is easily deduced.
\end{proof} 
A particular case is when $S_n$ is a rank-one perturbation of a matrix $S_n'$ which has no outlier and whose empirical spectral measure converges towards $\mu_{\mathrm{MP},\alpha} \boxtimes \mu_\Sigma$. In that setting, the interlacing property implies that $k_n=1$ for all $n$ sufficiently large in Corollary \ref{coro: BBP Wishart case}, meaning that $S_n$ has only one outlier which converges towards $x_{\alpha,\theta}$ and whose associated eigenvector has a square projection in the direction of the spike which converges towards $\frac{x_{\alpha,\theta}F(1/x_{\alpha,\theta})}{ F'(1/x_{\alpha,\theta})}$.

Before stating our second result, which is concerned with the projection of non-outlier eigenvectors onto the direction of the spike, we need the following Proposition.

\begin{prop}
$\mu_{\alpha,\Sigma,\theta}$ is absolutely continuous with respect to the Lebesgue measure on $\mathrm{Supp}(\mu_{\mathrm{MP},\alpha} \boxtimes \mu_{\Sigma} ) \setminus \{0\}$.
\end{prop}
\begin{proof}
Let $x \in \mathrm{Supp}(\mu_{\mathrm{MP},\alpha} \boxtimes \mu_{\Sigma} ) \setminus \{0\}$. The work of Cho\"i and Silverstein \cite{MR1345541} ensures that $\mu_{\mathrm{MP},\alpha} \boxtimes \mu_{\Sigma}$ is absolutely continuous with respect to the Lebesgue measure at $x$. Then, the inverse formula for the Stieltjes transform:
\begin{equation}  \frac{\mathrm{d\mu_{sc,A,\theta}}(x)}{\mathrm{d }x} =  \frac{1}{\pi} \lim\limits_{ t \rightarrow 0^+} \Im(s_{\mu_{sc,A,\theta}}(x+it) ),
\label{eq: Stielt inv Wishart} 
\end{equation} 
combined with Equation \eqref{eq: Stieltjes transform Wishart case}, implies the Proposition.
\end{proof}

Let $f_{\alpha,\Sigma}$ and $f_{\alpha,\Sigma,\theta}$ be the respective densities of $\mu_{\mathrm{MP},\alpha} \boxtimes \mu_\Sigma$ and $\mu_{\alpha, \Sigma, \theta}$ on $\mathrm{Supp}(\mu_{\mathrm{MP},\alpha} \boxtimes \mu_{\Sigma})$. It turns out that the averaged square projections of the non-outlier eigenvectors associated to eigenvalues in the vicinity of $x \in \mathrm{Supp}(\mu_{\mathrm{MP},\alpha} \boxtimes \mu_{\Sigma} ) $ converges towards the ratio of these two densities.

\begin{theorem}\label{theo: square proj Wishart general case}
Let $x\in \mathrm{Supp}(\mu_{\mathrm{MP},\alpha} \boxtimes \mu_{\Sigma} )  \setminus \{0 \} $ be such that $f_{\alpha, \Sigma}(x) > 0$. Let $\varepsilon_n$ be a sequence that satisfies $n^\delta / \sqrt{n} <\! \!< \varepsilon_n <\! \!< 1$ for some $\delta>0$. Then, for every $t>0$, if $\mathcal{I}^{(n)}_{\varepsilon_n}(x) = \left\{ 1 \leq i \leq n: \, | \lambda_i^{(n)} - x | \leq \varepsilon_n   \right\}$:
\[  \PP \left( \left|  \frac{n}{|\mathcal{I}^{(n)}_{\varepsilon_n}(x)|} \sum\limits_{i \in \mathcal{I}^{(n)}_{\varepsilon_n}(x)} \left| \langle \phi^{(n)}_i,v^{(n)}_1 \rangle \right|^2 - \frac{f_{\alpha,\Sigma,\theta}(x)}{f_{\alpha,\Sigma}(x)} \right| > t    \right)  \underset{n \rightarrow + \infty}{ \longrightarrow} 0.  \]
\end{theorem}

As in the Wigner case, Theorem \ref{theo: square proj Wishart general case} can be seen as a generalization of a result of \cite{MR2834718}, where the authors obtained the asymptotic behavior of the overlaps $|\langle \phi_i^{(n)}, v_j^{(n)} \rangle |^2$ after taking average over eigenvectors $\phi_i^{(n)}$'s (resp. $v_j^{(n)}$'s) with associated eigenvalues $\lambda_i^{(n)}$'s belonging to a {\it macroscopic} proportion of $\mathrm{Supp}(\mu_{\mathrm{MP},\alpha} \boxtimes \mu_\Sigma)$ (resp. $\mu_\Sigma$), by taking taking functions of the type $g= \mathbf{1}_{[\gamma,+\infty)}$ in the statistics introduced in Remark \ref{remark: LD}. Indeed, when $\theta \in \mathrm{Supp}( \mu_\Sigma )$, Theorem \ref{theo: square proj Wishart general case} is a {\it microscopic} version of the result of \cite[Theorem 3]{MR2834718}. To obtain their formula it suffices to remark that, if $a$ and $b$ are the real and imaginary parts of $1 - \alpha - zs_{\mu_{\mathrm{MP},\alpha} \boxtimes \mu_\Sigma}(z)$, one can rewrite, using Equation \eqref{eq: Stielt inv Wishart}:
\[ \frac{f_{\alpha,\Sigma,\theta}(x)}{f_{\alpha,\Sigma}(x)} = \frac{x \theta }{(a \theta - x)^2 + \theta^2 b^2},   \]
When $\theta \notin \mathrm{Supp}(\mu_\Sigma)$, the techniques of \cite{MR2834718} provide no information on the overlaps as it only gives access to $n^{-1}s_{\mu_{(S_n,v_1^{(n)})}}(z)$, which converges to zero, whereas the spectral measure approach still works.

\subsection{Rank-one perturbation of the Marchenko-Pastur law}\label{subsec: Wishart explicit}
In the peculiar case where $\gamma^{(n)}_2 = \cdots = \gamma^{(n)}_n = 1$ for all $n\geq 1$, $S_n$ is a rank-one perturbation of a classical Wishart matrix. The limiting spectrum of the perturbation is $\mu_{\Sigma}=\delta_1$ and almost surely, $\mu_{S_n}$ weakly converges towards the Marchenko-Pastur law $\mu_{\mathrm{MP},\alpha}$. All previous results have now a more explicit formulation. First, the limit of the spectral measure in the direction of the spike can be identified.
\begin{prop}\label{prop: Wishart case explicit comput of limiting measure}
In probability, $\mu_{(S_n,v^{(n)}_1)}$ converges towards:
\[ \mu_{\mathrm{MP},\alpha,\theta}(\mathrm{d}x) = \frac{\theta \sqrt{(b-x)(x-a)}}{2 \pi x(x(1-\theta) + \theta (\alpha\theta - \alpha + 1))} \mathbf{1}_{(a,b)}(x)\mathrm{d}x + c_{\alpha}\mathbf{1}_{\alpha<1}\delta_0(\mathrm{d}x) + d_{\alpha,\theta} \mathbf{1}_{ |\theta-1|>\frac{1}{ \sqrt{\alpha} }} \delta_{x_{\alpha,\theta}}(\mathrm{d}x),   \]
where $c_{\alpha} = \frac{1-\alpha}{\alpha(\theta-1)+1}$, $d_{\alpha,\theta} =  \frac{ 1- \frac{1}{\alpha(\theta-1)^2} }{  1 + \frac{1}{\alpha(\theta-1)} }$ and $x_{\alpha,\theta} = \frac{\theta(\alpha \theta - \alpha + 1)}{\theta - 1}$.
\end{prop}
A consequence of \cite{MR1617051} is that $n^{-1}X_nX_n^T$ has no outlier. Since $S_n$ is a rank one perturbation of this matrix, the discussion following Corollary \ref{coro: BBP Wishart case} implies that $S_n$ has a single outlier.
\begin{corollary} The following holds:
\begin{enumerate}
\item If $\theta > 1 + 1/\sqrt{\alpha} $, then, in probability, $\lambda_1(S_n)\underset{n \rightarrow + \infty}{\longrightarrow} x_{\alpha,\theta} > b $ and $|\langle \phi^{(n)}_1,v^{(n)}_1 \rangle| \underset{n \rightarrow + \infty}{\longrightarrow} \sqrt{d_{\alpha,\theta}}$.
\item If $\theta < 1 - 1/\sqrt{\alpha} $, then, in probability, $\lambda_n(S_n) \underset{n \rightarrow + \infty}{\longrightarrow} x_{\alpha,\theta} < a$ and $|\langle \phi^{(n)}_n,v^{(n)}_1 \rangle| \underset{n \rightarrow + \infty}{\longrightarrow} \sqrt{d_{\alpha,\theta}} $.
\end{enumerate}
\end{corollary}
The ratio of the density of $\mu_{\mathrm{MP},\alpha,\theta}$ and $\mu_{\mathrm{MP},\alpha}$ is explicit and we obtain the following Theorem.
\begin{theorem}\label{theo: averaged explicit Wishart}
Let $x\in (a,b)$. Let $\varepsilon_n$ be a sequence that satisfies $n^\delta / \sqrt{n} <\! \!< \varepsilon_n < \! \! < 1$ for some $\delta >0$. Then, for every $t>0$,
\[  \PP \left( \left|  \frac{n}{|\mathcal{I}^{(n)}_{\varepsilon_n}(x)|} \sum\limits_{i \in \mathcal{I}^{(n)}_{\varepsilon_n}(x)} \left| \langle \phi^{(n)}_i,v^{(n)}_1 \rangle \right|^2 - \frac{\theta}{x(1-\theta) + \theta(\alpha \theta - \alpha +1)} \right| > t    \right)  \underset{n \rightarrow + \infty}{ \longrightarrow} 0.  \]
\end{theorem}

When $x$ tends to $b$, the limiting profile becomes $\theta / (1+ \sqrt{\alpha}(1-\theta)^2)$, in accordance with \cite[Theorem 2.20]{MR3449395} where the authors obtain a convergence of individual square-projections onto the direction of the spike towards chi-squared random variables with expectation $\theta / (1+ \sqrt{\alpha}(1-\theta)^2)$. A natural question would be to study an analog convergence in law in the bulk of the spectrum (for any $x \in (a,b)$). We do not pursue this issue here.

\begin{figure}[ht!]
\centering
\begin{subfigure}{.5\textwidth}
  \centering
  \includegraphics[scale=0.4]{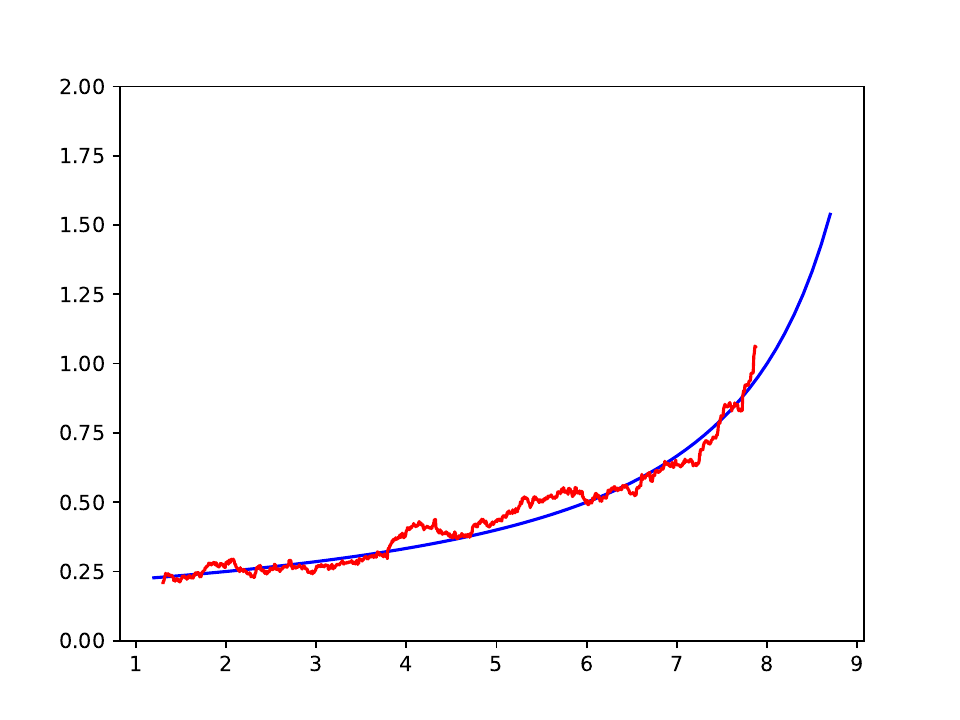} 
  \caption{}
  \label{fig:sfig11}
\end{subfigure}%
\begin{subfigure}{.5\textwidth}
  \centering
  \includegraphics[scale=0.4]{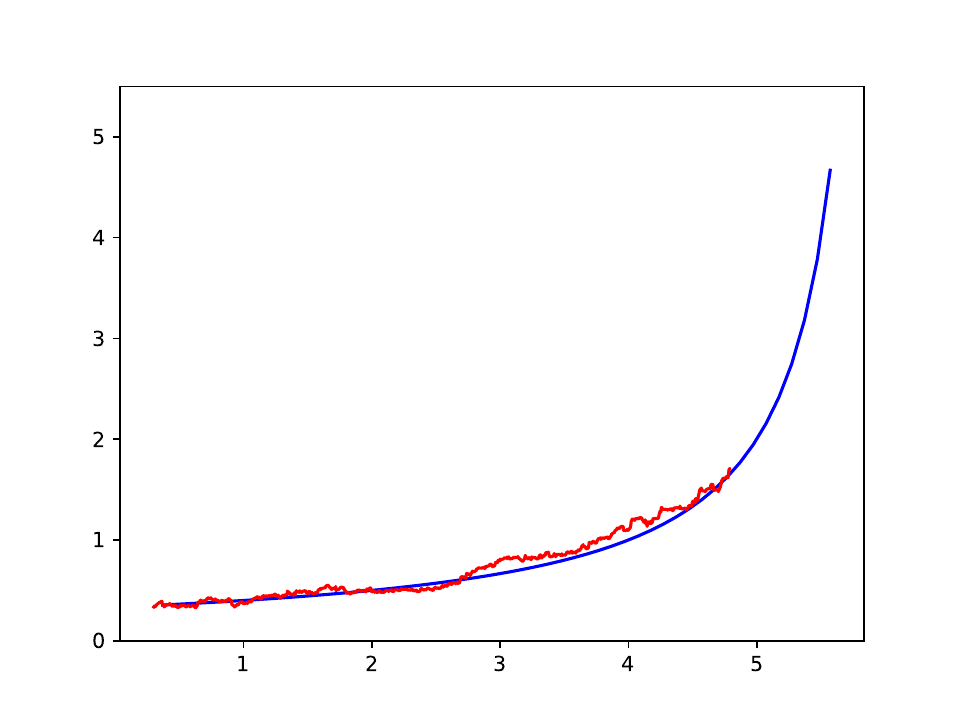} 
  \caption{}
  \label{fig:sfig21}
\end{subfigure}
\caption{{\footnotesize In red: simulations of the average squared projections around all locations $x\in(a,b)$ where we took average over interval of typical size $n^{0,1}/\sqrt{n}$, for a single matrix $S_n = n^{-1} \mathrm{Diag}(\sqrt{\theta},1, \ldots, 1) X_nX_n^T \mathrm{Diag}(\sqrt{\theta},1, \ldots, 1)$ where $X_n$ is gaussian rectangular of size $2000 \times 8000$ ($\alpha=4$) in case (a) and of size $2000 \times 4000$ ($\alpha=2$) in case (b). In each case $\theta = 2$. In blue: theoretical predictions.}}
\label{fig: averaged square proj1}
\end{figure}

\section{Identification of the limiting laws in rank-one perturbation cases}\label{sec: proof as convergence}
In this section we prove Proposition \ref{prop: Wigner case explicit comput of limiting measure} and \ref{prop: Wishart case explicit comput of limiting measure}. We use the following branch of the complex square-root:
\[  \sqrt{z} = \mathrm{sign}\left( \Im (z) \right) \frac{|z| + z}{ \sqrt{2(|z| + \Re (z))}}.   \]

\begin{proof}[Proof of Proposition \ref{prop: Wigner case explicit comput of limiting measure}]
The Stieltjes transform of the semicircle law is given by:
\[ s_{\mu_{sc}}(z) = \frac{-z + \sqrt{z^2-4}}{2} . \]
Therefore, using Equation \eqref{eq: Stieltjes transform eq Wigner case} and the fact that $\mu_A$ is in this case $\delta_0$:
\begin{align*}
s_{\mu_{sc,\theta}}(z) &= \frac{2}{2 \theta - z - \sqrt{z^2 -4}} \\
                       &= \frac{\sqrt{z^2 -4} +2 \theta -  z }{ 2 (\theta^2 - \theta z + 1) }.
\end{align*}
The absolutely continuous part of $\mu_{sc,\theta}$ is given by
\[  \frac{\mathrm{d}\mu_{sc,\theta} (x) }{\mathrm{d}x} = \frac{1}{\pi} \lim\limits_{z \rightarrow x^+} \Im (s_ {\mu_{sc,\theta}}(z)) = \frac{\sqrt{4-x^2}}{2 \pi (\theta^2 + 1 - \theta x)} \mathbf{1}_{ |x| \leq 2} \mathrm{d}x .  \]
The atom at $\theta + 1/\theta$ is given by the corresponding residue of $s_{\mu_{sc,\theta}}$:
\[ - \lim\limits_{z \rightarrow (\theta + 1/\theta)^+ } (z-\theta-1/\theta) s_{ \mu_{sc,\theta}}(z).   \]
By our choice of square-root, $\lim\limits_{z \rightarrow \theta + 1/ \theta } \sqrt{z^2 - 4} = |\theta - 1 |/\theta $ and one easily deduces:
\[ \mu_{sc,\theta}(\{\theta + 1/ \theta\}) =  \left\{ \begin{array}{cr}
0 & \text{if $|\theta| \leq 1$ } \\
1 - \frac{1}{\theta^2} & \text{if $|\theta| > 1$}.
\end{array}   \right.  \]

\end{proof}

\begin{proof}[Proof of Proposition \ref{prop: Wishart case explicit comput of limiting measure}]
Recall the expression of the Stieltjes transform of the Marchenko-Pastur law $\mu_{\mathrm{MP},\alpha} = \mu_{ \alpha,1}$:
\[ s_{\mu_{\mathrm{MP},\alpha}}(z) = \frac{\alpha-z-1 + \sqrt{(z-b)(z-a)}}{2z}.  \]
Substituting in Equation \eqref{eq: Stieltjes transform Wishart case}, we get
\begin{align}
s_{\mu_{\mathrm{MP},\alpha,\theta}}(z) &= \frac{-2}{2z - 2\theta(\alpha-1) + \theta(\alpha-1) - \theta z + \theta \sqrt{(z-b)(z-a)}}         \nonumber         \\
&= \frac{-2\left( 2z - \theta (\alpha - 1) - \theta z - \theta \sqrt{(z-b)(z-a)}   \right)}{\left( (2-\theta)z - \theta(\alpha-1)  \right)^2 - \theta^2(z-b)(z-a)} \nonumber \\
&= \frac{\theta \sqrt{(z-b)(z-a)} + z(\theta-2) +\theta(\alpha-1)}{2z\left( z(1-\theta) + \theta( \alpha \theta - \alpha +1)  \right)} \label{eq: expression Stransform Wishart explicit}.
\end{align}
This expression will allow us to obtain an explicit formula for $\mu_{\mathrm{MP},\alpha,\theta}$, through classical inversion results.

The absolutely continuous part of $\mu_{\mathrm{MP},\alpha,\theta}$ is given by:
\[  \frac{\mathrm{d}\mu_{\mathrm{MP},\alpha,\theta}}{\mathrm{d}x}(x) =  \frac{1}{\pi} \lim\limits_{ z \rightarrow x^+ } \Im(s_{\mu_{\mathrm{MP},\alpha,\theta}}(z)) = \frac{\theta \sqrt{(b-x)(x-a)}}{2 \pi x \left( x(1-\theta) + (\alpha \theta - \alpha + 1)  \right)} \mathbf{1}_{(a,b)}(x) .  \]
The atom of $\mu_{\mathrm{MP},\alpha,\theta}$ at zero is given by:
\[  -\lim\limits_{ \varepsilon \rightarrow 0^+} i\varepsilon s_{\mu_{\mathrm{MP},\alpha,\theta}}(i\varepsilon) = -\lim\limits_{ \varepsilon \rightarrow 0^+}  \frac{\sqrt{(i\varepsilon-b)(i\varepsilon-a)} + (\alpha-1)}{2(\alpha \theta - \alpha +1)}.  \]

By our choice of square-root that preserves the upper-half plane, $\sqrt{(i\varepsilon-b)(i\varepsilon-a)} \rightarrow -|ab| = -|\alpha-1|$ as $\varepsilon \rightarrow 0^+$. Therefore:
\[ \mu_{\mathrm{MP},\alpha,\theta}(\{0\}) = \frac{|\alpha - 1| - (\alpha - 1)   }{2( \alpha \theta - \alpha + 1 )} = \left\{ \begin{array}{cr}
0 & \text{if $\alpha \geq 1$ } \\
\frac{1-\alpha}{\alpha(\theta-1) + 1} & \text{if $\alpha<1$}.
\end{array}   \right. \]
Finally, let us compute the atom at $x_{\theta}=\frac{\theta( \alpha \theta - \alpha + 1)}{\theta - 1}$. It is given by:
\[ -\lim\limits_{ z \rightarrow x_{\theta}^+ } (z-x_{\theta}) s_{\mu_{\mathrm{MP},\alpha,\theta}}(z) = \frac{ \sqrt{ (x_{\theta}-b)(x_{\theta}-a)} + x_{\theta}(\theta-2) + \theta(\alpha-1)}{2(\theta-1)x_{\theta}} . \]
We use the following relations which are easily verified.
\begin{itemize}
\item $x_{\theta} - b = \frac{\alpha \left( (\theta - 1) - \frac{1}{\sqrt{\alpha}}  \right)^2}{\theta - 1}$,
\item $x_{\theta} - a = \frac{\alpha \left( (\theta - 1) + \frac{1}{\sqrt{\alpha}}  \right)^2}{\theta - 1}$.
\end{itemize}
As for the computation of the atom at zero:
\begin{align*} 
\lim\limits_{\varepsilon \rightarrow 0^+} \sqrt{(x_{\theta}+i\varepsilon-b)(x_{\theta}+i\varepsilon-a)} &= \mathrm{sign}(\theta-1) \frac{\alpha}{|\theta-1|} \left| (\theta-1)^2 - \frac{1}{\alpha}  \right|  \\
&= \frac{\alpha}{\theta-1} \left| (\theta-1)^2 - \frac{1}{\alpha}  \right|.
\end{align*}
Besides:
\begin{align*} 
x_{\theta}(\theta-2) + \theta(\alpha-1) &= \frac{\theta(\alpha \theta - \alpha + 1) (\theta - 2) + \theta (\alpha -1)(\theta-1)}{\theta - 1}   \\
                                        &= \frac{\alpha \theta}{\theta-1}\left( (\theta-1)^2 - \frac{1}{\alpha}\right).
\end{align*}
One obtains:
\[  \mu_{\mathrm{MP},\alpha,\theta}( \{ x_{\alpha,\theta} \}) = \alpha \theta \frac{ \left| (\theta-1)^2 - \frac{1}{\alpha} \right| + \left( (\theta-1)^2 - \frac{1}{\alpha}   \right) }{2 \theta (\theta - 1) (\alpha \theta - \alpha + 1)} =  \left\{ \begin{array}{cr}
0 & \text{if $|\theta - 1| \leq 1/\sqrt{\alpha}$ } \\
\frac{\alpha \left( (\theta-1)^2 - \frac{1}{\alpha}  \right)}{(\theta- 1)(\alpha \theta - \alpha + 1) } & \text{if $|\theta - 1| > 1/\sqrt{\alpha}$}.
\end{array}   \right. \]
\end{proof}

\section{Convergence of the averaged square projections}\label{sec: proof averaged}

We only focus on the proof of Theorem \ref{theo: square proj Wishart general case} concerning the convergence of averaged square-projections into the direction of the spike in the Wishart setting. The proof of Theorem \ref{theo: square proj Wigner general case}, which concerns the Wigner setting, would follow the same reasoning, the only difference being the use of the local law of Theorem \ref{theo: local law Wigner} instead of the local law of Theorem \ref{theo: local law Wishart}. For the rest of this section, we fix $0<\delta<1/2$ and $(\varepsilon_n)_{n \geq 1}$ a sequence of real numbers such that $n^\delta / \sqrt{n} < \! \! < \varepsilon_n < \! \! < 1$. We also fix $x_0 >0$ such that $\mu_{\alpha, \Sigma,\theta}$ has a positive density at $x_0$. 

Let us explain the heuristic behind Theorem \ref{theo: square proj Wishart general case}. We will denote $I_{\varepsilon_n}(x_0) := [x_0 -\varepsilon_n, x_0 + \varepsilon_n]$. Recall that $f_{\alpha, \Sigma}$ and $f_{\alpha, \Sigma, \theta}$ are the respective densities of $\mu_{\mathrm{MP},\alpha} \boxtimes \mu_\Sigma$ and $\mu_{\alpha,\Sigma,\theta}$ on $\mathrm{Supp}\left( \mu_{\mathrm{MP},\alpha} \boxtimes \mu_\Sigma  \right)$. Then,
\[  \int_{I_{\varepsilon_n}(x_0)} \mathrm{d}\mu_{(S_n,v^{(n)}_1)}(x) \approx \int_{I_{\varepsilon_n}(x_0)} \mathrm{d}\mu_{\mathrm{MP},\alpha,\theta}(x) + o_1(1) \approx 2 \varepsilon_n f_{\alpha, \Sigma, \theta}(x_0) + o_1(1).   \]
On the other hand, if $\mu_{S_n} = n^{-1} \sum_{1 \leq i \leq n} \delta_{\lambda^{(n)}_i}$ denotes the empirical spectral measure of $S_n$:
\begin{align*}  \int_{I_{\varepsilon_n}(x_0)} \mathrm{d}\mu_{(S_n,v^{(n)}_1)}(x) = \sum\limits_{i \in \mathcal{I}^{(n)}_{\varepsilon_n}(x_0)} \left| \langle \phi_i,e_1 \rangle \right|^2 &= \left( \frac{n}{|\mathcal{I}^{(n)}_{\varepsilon_n}(x_0)|} \sum\limits_{i \in \mathcal{I}^{(n)}_{\varepsilon_n}(x_0)} \left| \langle \phi^{(n)}_i,v^{(n)}_1 \rangle \right|^2 \right) \times \int_{I_{\varepsilon_n}(x_0)} \mathrm{d}\mu_{S_n}(x) \\ 
                                                               &\approx \left( \frac{n}{|\mathcal{I}^{(n)}_{\varepsilon_n}(x_0)|} \sum\limits_{i \in \mathcal{I}^{(n)}_{\varepsilon_n}(x_0)} \left| \langle \phi_i,e_1 \rangle \right|^2 \right) \times \left( 2 \varepsilon_n f_{\alpha, \Sigma}(x_0) + o_2(1) \right),
\end{align*}
where we recall that $\mathcal{I}^{(n)}_{\varepsilon_n}(x_0) = \{ 1 \leq i \leq n: \, | \lambda^{(n)}_i -x_0 | \leq \varepsilon_n \}$. Theorem \ref{theo: square proj Wishart general case} would be proved if the errors $o_1(1)$ and $o_2(1)$ were explicit and of a smaller order than $\varepsilon_n$. The understanding of these errors is precisely the purpose of the so-called local laws that have been recently developed in random matrix theory. In the Wishart setting, it is given by the local law of Knowles and Yin stated in Theorem \ref{theo: local law Wishart}. 

The rest of this section makes the above heuristic rigorous. It combines a local law on the Stieltjes transform of $\mu_{(S_n,v_1^{(n)})}$ together with an approximation argument which allows to estimate a term of the form $\mu_{(S_n,v_1^{(n)})}(I_{\varepsilon_n}(x_0))$. The idea of the latter is to bound the indicator of $I_{\varepsilon_n}(x_0)$ by two smooth analytic functions which we now introduce. Let $(\omega_n)_{n\geq 1}$ be a sequence of real numbers such that 
\begin{equation}\label{eq:omega}
n^\delta / \sqrt{n} < \! \! < \omega_n < \! \! < \varepsilon_n.
\end{equation}
Let $\Psi$ be a smooth decreasing function such that $\Psi(x) \equiv 1$ on $(-\infty,0]$ and $\Psi(x) \equiv 0$ on $[1,+\infty)$. For all $n \geq 1$, we define:
\begin{equation}\label{eq:RegFunctions}
\begin{cases}
\phi_n^{-}(x) = \Psi \left( 1 +  \frac{x - x_0 - \varepsilon_n + \omega_n}{\omega_n}   \right) \Psi \left(  1 - \frac{x-x_0 + \varepsilon_n - \omega_n}{\omega_n} \right),\\
\phi_n^{+}(x) = \Psi \left(  \frac{x - x_0 - \varepsilon_n - \omega_n}{\omega_n}   \right) \Psi \left( - \frac{x-x_0 + \varepsilon_n + \omega_n}{\omega_n} \right).
\end{cases}
\end{equation}
With these definitions, it is easy to check the following properties.
\begin{enumerate}
\item the support of $\phi^{-}_n$ (resp. $\phi^+_n$) is included in $[x_0-\varepsilon_n+ \omega_n,x_0+\varepsilon- \omega_n]$ (resp. $[x_0 -\varepsilon_n- 2\omega_n,x_0 +\varepsilon_n+ 2\omega_n]$);
\item $\phi^-_n$ (resp. $\phi^+_n$) is constant equal to $1$ on $[x_0 -\varepsilon_n + 2 \omega_n, x_0 + \varepsilon_n - 2 \omega_n]$ (resp. $[x_0 -\varepsilon_n-\omega_n,x_0 +\varepsilon_n+\omega_n]$);
\item the supports of $(\phi_n^-)'$ and $(\phi_n^-)''$ (resp. $(\phi_n^+)'$ and $(\phi_n^+)''$) are included into $[x_0-\varepsilon_n+\omega_n,x_0-\varepsilon_n+2 \omega_n] \cup [x_0+\varepsilon_n-2\omega_n,x_0+\varepsilon_n - \omega_n]$ (resp. $[x_0-\varepsilon_n- 2\omega_n,x_0-\varepsilon_n - \omega_n] \cup [x_0+\varepsilon_n + \omega_n,x_0+\varepsilon_n + 2\omega_n]$);
\item $||(\phi^-_n)'||_{\infty} = || (\phi^+_n)' ||_{\infty} = O(1/\omega_n)$ and $||(\phi^-_n)''||_{\infty} = || (\phi^+_n)'' ||_{\infty} = O(1/\omega_n^2)$.
\end{enumerate}
By construction,
\begin{equation}\label{eq:IndicBound}
\begin{cases} 
\int_{\mathbf{R}} \phi^-_n(\lambda) \mathrm{d}\mu_{(S_n,v^{(n)}_1)}(\lambda) \leq \int_{I_{\varepsilon_n}(x_0)} \mathrm{d}\mu_{(S_n,v^{(n)}_1)}(\lambda) \leq \int_{\mathbf{R}} \phi^+_n(\lambda) \mathrm{d}\mu_{(S_n,v^{(n)}_1)}(\lambda), \\
\int_{\mathbf{R}} \phi^-_n(\lambda) \mathrm{d}\mu_{S_n}(\lambda) \leq \int_{I_{\varepsilon_n}(x_0)} \mathrm{d}\mu_{(S_n,v^{(n)}_1)}(\lambda) \leq \int_{\mathbf{R}} \phi^+_n(\lambda) \mathrm{d}\mu_{S_n}(\lambda).
\end{cases}  
\end{equation}
The main result of this section is an estimate on both sides of the above inequalities.

\begin{lemme}\label{lem:Helff}
Let $\epsilon \in \{-,+\}$. There exists $D>0$ such that, with probability at least $1-n^{-D}$,
\begin{equation}\label{eq:Helff1}
\left| \int_{\mathbf{R}} \phi_n^\epsilon(\lambda) \mathrm{d} \left( \mu_{(S_n,v_1^{(n)})} - \mu_{\alpha,\Sigma,\theta} \right) (\lambda)\right| = O\left( \omega_n \right)
\end{equation}
and
\begin{equation}\label{eq:Helff2}
\left| \int_{\mathbf{R}} \phi_n^\epsilon(\lambda) \mathrm{d} \left( \mu_{S_n} - \mu_{\mathrm{MP},\alpha} \right) (\lambda)\right| = O\left( \omega_n \right).
\end{equation}
\end{lemme} 
Before giving the proof of Lemma \ref{lem:Helff}, let us explain how it leads to Theorem \ref{theo: square proj Wishart general case}.
\begin{proof}[Proof of Theorem \ref{theo: square proj Wishart general case}]
Combining estimates \eqref{eq:Helff1} and \eqref{eq:Helff2} with Equation \ref{eq:IndicBound}, for any $\varepsilon >0$, there exists $D>0$ such that with probability at least $1 - n^{-D}$:
\begin{equation}  \int_{ I_{\varepsilon_n}(x_0) } \mathrm{d}\mu_{(S_n,v_1^{(n)})}(x) = 2 \varepsilon_n f_{ \alpha, \Sigma, \theta } (x_0) + O\left(  \omega_n  \right)
\label{eq: small scale spectral measure} 
\end{equation}
and
\begin{equation}  \int_{ I_{\varepsilon_n}(x_0) } \mathrm{d}\mu_{S_n}(x) = 2 \varepsilon_n f_{ \alpha, \Sigma } (x_0) + O\left(  \omega_n \right). 
\label{eq: small scale empirical measure}
\end{equation}
Therefore, since $\varepsilon_n > \! \! > n^{-1/2}$, for all $\varepsilon>0$, there exists $D>0$ such that with probability at least $1-n^{-D}$,
\begin{align*}  
\frac{n}{\left| \mathcal{I}_{\varepsilon_n}(x_0) \right|} \sum\limits_{ i \in \mathcal{I}_{\varepsilon_n}(x_0)} \left| \langle \phi_i , e_1 \rangle  \right|^2  
&=  \frac{ \int_{ I_{\varepsilon_n}(x_0) } \mathrm{d}\mu_{(S_n,v_1^{(n)})}(\lambda) }{  \int_{ I_{\varepsilon_n}(x_0) } \mathrm{d}\mu_{S_n}(\lambda)  } \\
&= \frac{f_{ \alpha, \Sigma, \theta } (x_0)}{f_{ \alpha, \Sigma } (x_0)} + O\left(  \frac{\omega_n}{\varepsilon_n}  \right).
\end{align*}
This ends the proof of Theorem \ref{theo: square proj Wishart general case} since $\varepsilon_n > \! \! > \omega_n$.
\end{proof}

We now turn to the proof of Lemma \ref{lem:Helff}. We will use a local law on the Stieltjes transform of $\mu_{(S_n,v_1^{(n)})}$ which can be obtained by taking $v=w=v_1^{(n)}$ in Equation \eqref{eq: estimate local law Wishart}. Recall the following definitions:
\[ \mathcal{D}_n^{(\tau)}(x,c) = \left\{ z \in \mathbf{C}, \, x-c \leq E \leq x+c, \, n^{-1 + \tau} \leq \eta \leq \tau^{-1}    \right\} \]
and
\[ \psi_n (z)= \sqrt{ \frac{\Im \left( s_{ \mu_{\mathrm{MP},\alpha} \boxtimes \mu_\Sigma } (z) \right) }{n \eta} } + \frac{1}{n \eta} .\]
Since $x_0$ is such that $f_{\alpha, \Sigma}(x_0)>0$, there exists $c>0$ small enough such that Equation \eqref{eq:positivedens2} is satisfied. We fix such a $c$ for the rest of this section. Then, Theorem \ref{theo: local law Wishart} translates into the following result.

\begin{Theo}{\bf \cite[Corollary 3.9]{MR3704770}} \label{th:LocLawStieltjes}
For any $\tau >0$, uniformly in $z \in \mathcal{D}_n^{(\tau)}(x_0,c)$, for any $\varepsilon > 0$, there exists $D >0$ such that
\begin{equation}
\PP\left( \left| s_{(S_n,v_1^{(n)})}(z) - s_{\mu_{\alpha,\Sigma,\theta}}(z)  \right| \geq n^\varepsilon       \psi(z)  \right) \leq \frac{1}{n^D}.
\label{eq:loclaw}
\end{equation}
\end{Theo}

In the proof of Lemma \ref{lem:Helff}, we will use the following classical estimate on the error function.

\begin{lemme}\label{lem:psi}
For all $x\in \mathbf{R}$ such that $f_{\alpha,\Sigma}(x)>0$,
\begin{equation}\label{eq:psi1}
|\psi_n(x+iy)| = O\left( \sqrt{ \frac{\Im \left( s_{ \mu_{\mathrm{MP},\alpha} \boxtimes \mu_\Sigma } (x+iy) \right) }{n y} } \right).
\end{equation}
In particular, for all $1 \leq |y| \leq 2$, 
\begin{equation}\label{eq:psi2}
|\psi_n(x+iy)| = O\left( n^{-1/2} \right).
\end{equation}
\end{lemme}
\begin{proof}
Let $x \in \mathbf{R}$ be such that $f_{\alpha,\Sigma}(x)>0$. Then, since
\[ f_{\alpha,\Sigma}(x) = \frac{1}{\pi} \lim\limits_{t \rightarrow 0^+} \Im \left( s_{ \mu_{\mathrm{MP},\alpha} \boxtimes \mu_\Sigma } (x+it) \right),  \]
the leading term in the expression of $\psi_n(x+iy)$ is $\sqrt{ \frac{\Im \left( s_{ \mu_{\mathrm{MP},\alpha} \boxtimes \mu_\Sigma } (x+iy) \right) }{n y} }$, which proves \eqref{eq:psi1}. This latter term is of order $n^{-1/2}$ whenever $y$ is bounded away from $0$, which implies \eqref{eq:psi2}.
\end{proof}

We are now ready to prove Lemma \ref{lem:Helff}. The strategy is based on the Helffer-Sj\"ostrand formula, which allows to translate an estimate of the form \eqref{eq:loclaw} into an estimate on sufficiently regular functions integrated against $\mu_{(S_n,v^{(n)}_1)}$. Although the argument is standard and can be found for the empirical spectral measure of a Wigner matrix in the survey of Benaych-Georges and Knowles \cite{MR3792624}, we choose to provide the details as it has not been done for the spectral measures. A similar argument is also present in the work of Benaych-Georges, Enriquez and Micha{\"\i}l \cite{benaych2018eigenvectors}. 

\begin{proof}[Proof of Lemma \ref{lem:Helff}] 
We first prove estimate \eqref{eq:Helff1} that corresponds to the spectral measure and then explain how to adapt the proof for the second inequality \eqref{eq:Helff2} which corresponds to the empirical spectral measure. We only focus on the case $\epsilon = +$ because the case $\epsilon = -$ follows from the same argument. In order to lighten notations, we denote $\phi_n =\phi_n^+$.

For all $x \in \mathbf{R}$, by the the Helffer-Sj\"ostrand formula (see \cite[Proposition $C.1$]{MR3792624} ):
\begin{equation}  \phi_n(x) =  \int_{\mathbf{C}} \frac{ \overline{\partial}\left( \tilde{\phi}_n(z) \chi(z)  \right) }{x-z} \mathrm{d}z, 
\label{eq: Helffer-Sjostrand}
\end{equation}
where:
\begin{itemize}
\item $\chi$ is a smooth symmetric cutoff function that equals $1$ on $[-1,1]$ and $0$ outside $[-2,2]$;
\item $\tilde{\phi}_n$ is the quasi-analytic extension of degree $1$ of $\phi_n$, defined by $\tilde{\phi}_n(x+iy) = \phi_n(x) + iy \phi'_n(x)$;
\item $\overline{\partial} = \frac{1}{2}\left( \partial_n + i \partial_y \right)$.
\end{itemize}
Let us define
\[ \hat{\mu}_n := \mu_{(S_n,v_1^{(n)})} - \mu_{\alpha,\Sigma,\theta} \]
and its Stieltjes transform $\hat{s}_n := s_{\mu_{(S_n,v_1^{(n)})}} - s_{\mu_{\alpha,\Sigma,\theta}}$. Equation \eqref{eq: Helffer-Sjostrand} leads to:
\begin{align*} 
\int_{\mathbf{R}} \phi_n(\lambda) \mathrm{d}\hat{\mu}_n(\lambda) = &\phantom{+b} \frac{i}{2 \pi} \int_{x \in \mathbf{R}} \int_{y \in \mathbf{R}} \phi''_n(x) y \chi (y) \hat{s}_n(x+iy) \mathrm{d}x \mathrm{d}y \\
&+ \frac{i}{2 \pi} \int_{x \in \mathbf{R}} \int_{y \in \mathbf{R}} \left( \phi_n(x) + iy\phi'_n(x)  \right) \chi'(y) \hat{s}_n(x+iy) \mathrm{d}x \mathrm{d}y.
\end{align*}
Remark that the right-hand-side is real so that
\begin{align} 
\int_{\mathbf{R}} \phi_n(\lambda) \mathrm{d}\hat{\mu}_n(\lambda) \leq &\phantom{+b} \frac{-1}{2 \pi} \int_{x \in \mathbf{R}} \int_{|y| \leq \omega_n} \phi''_n(x) y \chi (y) \Im \left(\hat{s}_n (x+iy)\right) \mathrm{d}x \mathrm{d}y \label{eq:Term1} \\
&+ \frac{-1}{2 \pi} \int_{x \in \mathbf{R}} \int_{|y| \geq \omega_n} \phi''_n(x) y \chi (y) \Im \left(\hat{s}_n (x+iy) \right) \mathrm{d}x \mathrm{d}y \label{eq:Term2} \\
&+ \left| \frac{1}{2 \pi} \int_{x \in \mathbf{R}} \int_{y \in \mathbf{R}} \left( \phi_n(x) + iy\phi'_n(x)  \right) \chi'(y) \hat{s}_n (x+iy) \mathrm{d}x \mathrm{d}y \right| \label{eq:Term3} .
\end{align}
We now estimate all of the three terms of the right-hand side. In what follows, we fix $\tau > 0$ and $\varepsilon >0$ and we argue on the event
\begin{equation} \label{eq:EventE}
E_\varepsilon := \left\{ \forall z \in \mathcal{D}_n^{(\tau)}(x_0,c), \,  \left| s_{(S_n,v_1^{(n)})}(z) - s_{\mu_{\alpha,\Sigma,\theta}}(z)  \right| \leq n^\varepsilon \psi_n (z)  \right\} .
\end{equation}
Note that by the local law \eqref{eq:loclaw}, there exists $D>0$ such that $\PP(E) \geq 1 - n^{-D}$. 
\paragraph*{Estimation of the term \eqref{eq:Term1}.}
Recall that, from the definition of $\phi_n$, given in Equation \eqref{eq:RegFunctions}, the support of $\phi_n''$ is contained in 
\begin{equation}\label{eq:defn:intervals}
I_n := [x_0-\varepsilon_n- 2\omega_n,x_0-\varepsilon_n - \omega_n] \cup [x_0+\varepsilon_n + \omega_n,x_0+\varepsilon_n + 2\omega_n].
\end{equation}
Moreover, since $\omega_n <\! \!<1$, $\chi(y)=1$ when $|y| \leq \omega_n$. Therefore:
\begin{align} 
\left| \hspace{-0.3cm} \phantom{\frac{1}{1}} \right. &  \frac{-1}{2 \pi}   \left. \int_{x \in \mathbf{R}} \int_{|y| \leq \omega_n} \phi''_n(x) y \chi (y) \Im \left(\hat{s}_n (x+iy)\right) \mathrm{d}x \mathrm{d}y  \right| \nonumber \\
&\leq \frac{1}{2 \pi}  \int_{x \in I_n} \int_{0 \leq y \leq \omega_n} \left|\phi''_n(x) y \Im \left(\hat{s}_n (x+iy)\right) \right| \mathrm{d}x \mathrm{d}y + \frac{1}{2 \pi} \int_{x \in I_n} \int_{-\omega_n \leq y < 0} \left|\phi''_n(x) y \Im \left(\hat{s}_n (x+iy)\right) \right| \mathrm{d}x \mathrm{d}y  \label{eq:FirstTerm}
\end{align}
We only treat in details the first term of \eqref{eq:FirstTerm} as the second one can be analyzed similarly.

Let $x \in I_n$. The function $y \mapsto y \Im \left(  s_{\mu_{(S_n,v_1^{(n)})} }(x+iy) \right)$ is non-decreasing, which implies that: 
\begin{equation} \label{eq:MonotonyIm}
\forall 0 \leq y \leq \omega_n, \quad y \Im \left(\hat{s}_n (x+iy)\right)   \leq  y \Im \left( s_{\mu_{(S_n,v_1^{(n)})} } (x+iy)\right)  \leq \omega_n \Im \left( s_{\mu_{(S_n,v_1^{(n)})} } (x+i\omega_n)\right). 
\end{equation}
By assumptions, the point $x_0 \in \mathbf{R}$ is such that $f_{\alpha,\Sigma,\theta}(x_0)>0$. Since $\lim\limits_{t \rightarrow 0^+} \Im \left( s_{\mu_{\alpha,\Sigma,\theta}}(x_0 + it  \right) =  \pi f_{\alpha,\Sigma,\theta}(x_0)$, this implies that there exists a constant $\tilde{C}>0$ such that, for large enough $n$:
\[  \forall x \in I_n, \, \forall 0 \leq y \leq \omega_n, \quad  \Im \left( s_{\mu_{\alpha,\Sigma,\theta}}(x + iy)  \right) \leq \tilde{C}.   \]
Therefore, since we are working on the event $E_\varepsilon$ introduced in \eqref{eq:EventE}, uniformly in $x \in I_n$:
\begin{equation} \label{eq:BoundStielt}
\Im \left(  s_{\mu_{(S_n,v_1^{(n)})} } (x+i \omega_n)\right) \leq  n^\varepsilon \psi_n(x+i \omega_n) + \tilde{C}.  
\end{equation}
By Equation \eqref{eq:psi1} of Lemma \ref{lem:psi}, $n^\varepsilon \psi_n(x+i \omega_n)$ converges to zero as $n$ tends to infinity. Therefore, by combining Inequalities \eqref{eq:MonotonyIm} and \eqref{eq:BoundStielt}, we deduce the existence of some constant $C>0$ such that:
\begin{equation} \label{eq:BoundFirstTerm}
\int_{x \in I_n} \int_{0 \leq y \leq \omega_n} \left|\phi''_n(x) y \Im \left(\hat{s}_n (x+iy)\right) \right| \mathrm{d}x \mathrm{d}y \leq  C \omega_n \int_{x \in I_n} \int_{0 \leq y \leq \omega_n} |\phi''_n(x)|  \mathrm{d}x \mathrm{d}y.
\end{equation}
Finally, using that $||\phi_n''||_\infty = O(1/\omega_n^2)$ and $\int_{I_n} \mathrm{d}x = 2 \omega_n$ in \eqref{eq:BoundFirstTerm} yields:
\begin{equation} \label{eq:PreFinalBoundFirstTerm}
\int_{x \in I_n} \int_{0 \leq y \leq \omega_n} \left|\phi''_n(x) y \Im \left(\hat{s}_n (x+iy)\right) \right| \mathrm{d}x \mathrm{d}y \leq O \left( \omega_n \right).
\end{equation}
As already mentioned, the same argument implies that the bound \eqref{eq:PreFinalBoundFirstTerm} also holds if the domain of integration was $I_n \times[-\omega_n,0)$. Hence, we proved that, on the event $E_\varepsilon$,
\begin{equation} \label{eq:bound1}
\left|  \frac{-1}{2 \pi}  \int_{x \in \mathbf{R}} \int_{|y| \leq \omega_n} \phi''_n(x) y \chi (y) \Im \left(\hat{s}_n (x+iy)\right) \mathrm{d}x \mathrm{d}y  \right| = O(\omega_n).
\end{equation}

\paragraph*{Estimation of the term \eqref{eq:Term2}.}

We first decompose the term according to the sign of $y$.
\begin{align}
& \frac{-1}{2 \pi}  \int_{x \in \mathbf{R}}  \int_{|y| \geq \omega_n} \phi''_n(x) y \chi (y) \Im \left(\hat{s}_n (x+iy) \right) \mathrm{d}x \mathrm{d}y \nonumber  \\
 &= \frac{-1}{2 \pi} \int_{x \in \mathbf{R}}  \int_{y\leq -\omega_n} \phi''_n(x) y \chi (y) \Im \left(\hat{s}_n (x+iy) \right) \mathrm{d}x \mathrm{d}y + \frac{-1}{2 \pi} \int_{x \in \mathbf{R}}  \int_{y \geq \omega_n} \phi''_n(x) y \chi (y) \Im \left(\hat{s}_n (x+iy) \right) \mathrm{d}x \mathrm{d}y \label{eq:SecondTerm}.
\end{align}
The two terms on the right-hand side of \eqref{eq:SecondTerm} can be analyzed in the same way, so that we only focus on the case where $y \geq \omega_n$.

Differentiating with respect to $x$ and $y$ and using that $\partial_x \Im ( \hat{s}_n (x+iy) ) = - \partial_y \Re ( \hat{s}_n (x+iy))$, we obtain:
\begin{align}
\frac{-1}{2 \pi} & \int_{x \in \mathbf{R}} \int_{y \geq \omega_n} \phi''_n(x) y \chi (y) \Im \left(\hat{s}_n (x+iy) \right) \mathrm{d}x \mathrm{d}y \nonumber \\
                 &= \frac{-1}{2 \pi} \int_{y \geq \omega_n} y \chi(y) \mathrm{d}y \times (-1) \int_{x \in \mathbf{R}} \phi_n'(x) \partial_x \Im \left( \hat{s}_n(x+iy)  \right) \mathrm{d}x \nonumber  \\
                 &= \frac{-1}{2 \pi} \int_{ x \in \mathbf{R}}  \phi_n'(x) \mathrm{d}x \left(  \omega_n \Im \left( \hat{s}_n(x+i \omega_n)   \right)  - \int_{y > \omega_n} (y \chi'(y) + \chi(y))  \Im \left( \hat{s}_n(x+iy)   \right)  \mathrm{d}y  \right) \nonumber \\
                 &= \frac{-1}{2 \pi} \int_{ x \in \mathbf{R}} \phi_n'(x)  \omega_n \Im \left( \hat{s}_n(x+i \omega_n)   \right) \mathrm{d}x \label{eq:align:21} \\
                 &\phantom{=} \hspace{2.0cm}+  \frac{1}{2 \pi} \int_{ x \in \mathbf{R}} \int_{y > \omega_n} \phi_n'(x) y \chi'(y) \Im \left( \hat{s}_n(x+iy)   \right) \mathrm{d}y \mathrm{d}x \label{eq:align:22} \\ 
                 &\phantom{=} \hspace{2.0cm}+ \frac{1}{2 \pi} \int_{ x \in \mathbf{R}} \int_{y > \omega_n} \phi_n'(x) \chi(y) \Im \left( \hat{s}_n(x+iy)   \right) \mathrm{d}y \mathrm{d}x \label{eq:align:23} .
\end{align}
The first term \eqref{eq:align:21} can be bounded as follows, recalling that $\phi_n'$ is supported on $I_n$ (defined in Equation \eqref{eq:defn:intervals}):
\begin{equation*}   
\left| \frac{-1}{2 \pi} \int_{ x \in \mathbf{R}}  \phi_n'(x) \omega_n \Im \left( \hat{s}_n(x+i \omega_n)   \right) \mathrm{d}x  \right| 
\leq \frac{1}{2 \pi} ||\phi_n'||_\infty \omega_n \int_{I_n} \left| \Im \left( \hat{s}_n(x+i \omega_n)   \right) \right| \mathrm{d}x.
\end{equation*}
By the definition of $\phi_n$, it holds that $||\phi_n'||_\infty = O(1/\omega_n)$. Combining this with inequality \eqref{eq:BoundStielt}, which holds uniformly in $x \in I_n$, we get that:
\begin{equation} \label{eq:bound21}   
\left| \frac{-1}{2 \pi} \int_{ x \in \mathbf{R}}  \phi_n'(x) \omega_n \Im \left( \hat{s}_n(x+i \omega_n)   \right) \mathrm{d}x  \right| = O(\omega_n).
\end{equation}

The second term \eqref{eq:align:22} can be bounded as follows, using that $\chi'$ is supported on $[-2,-1] \cup [1,2]$:
\[ \left| \frac{1}{2 \pi} \int_{ x \in \mathbf{R}} \int_{ y \geq \omega_n} \phi_n'(x)y \chi'(y)   \Im \left( \hat{s}_n(x+iy)   \right) \mathrm{d}y \mathrm{d}x \right| \leq \frac{1}{2 \pi} ||\phi_n'||_\infty \int_{ x \in I_n } \int_{1\leq  y \leq 2} y \chi'(y) \Im \left( \hat{s}_n(x+iy)   \right) \mathrm{d}x \mathrm{d}y.  \]
Since we are working on the event $E_\varepsilon$, $|\Im \left( \hat{s}_n(x+iy)   \right)| \leq n^\varepsilon \psi_n(x+iy)$. Therefore, using Lemma \ref{lem:psi}:
\begin{equation} \label{eq:bound22}
\left| \frac{1}{2 \pi} \int_{ x \in \mathbf{R}} \int_{ y \geq \omega_n} \phi_n'(x)y \chi'(y)   \Im \left( \hat{s}_n(x+iy)   \right) \mathrm{d}y \mathrm{d}x \right| = O\left(  \frac{n^\varepsilon}{\sqrt{n}}  \right).
\end{equation}

Finally, we bound the last term \eqref{eq:align:23}. Since $\phi_n'$ is supported on $I_n$, we can replace the integral over $\mathbf{R}$ by an integral over $I_n$. Moreover, using \eqref{eq:psi1}, we get that:
\begin{align} 
\left| \frac{1}{2 \pi} \int_{ x \in \mathbf{R}} \int_{y > \omega_n} \phi_n'(x) \chi(y) \Im \left( \hat{s}_n(x+iy)   \right) \mathrm{d}y \mathrm{d}x \right| 
&\leq \frac{1}{2 \pi}||\phi_n'||_\infty \int_{I_n} \int_{\omega_n}^2 \frac{n^\varepsilon}{\sqrt{n y}} \mathrm{d}y \mathrm{d}x  \nonumber \\ 
&= O \left( \frac{n^\varepsilon}{\sqrt{n}} \right) \label{eq:bound23}.
\end{align}

Hence, putting together \eqref{eq:bound21}, \eqref{eq:bound22} and \eqref{eq:bound23}, we proved that, on the event $E_\varepsilon$:
\begin{equation}  \left| \frac{-1}{2 \pi} \int_{x \in \mathbf{R}} \int_{|y| \geq \omega_n} \phi''_n(x) y \chi (y) \Im \left(\hat{s}_n (x+iy) \right) \mathrm{d}x \mathrm{d}y \right| = O\left( \max \left\{ \frac{n^{\varepsilon} }{n^{1/2} }, \omega_n \right\} \right) 
\label{eq:bound2}
\end{equation}

\paragraph*{Estimation of the term \eqref{eq:Term3}.}
Since $\chi'$ is symmetric and supported on $[-2,-1] \cup [1,2]$, there exists some constant $C>0$ such that:
\begin{multline*}
\left| \frac{1}{2 \pi} \int_{x \in \mathbf{R}} \int_{y \in \mathbf{R}} \left( \phi_n(x) + iy\phi'_n(x)  \right) \chi'(y) \hat{s}_n (x+iy) \mathrm{d}x \mathrm{d}y \right| \\
 \leq \frac{C}{2 \pi} \int_{x \in [x_0 - \varepsilon_n - 2 \omega_n, x_0 + \varepsilon_n + 2 \omega_n]} \int_{y \in [-2,-1] \cup [1,2]} \left( 1 +  ||\phi'_n||_\infty  \right) \left| \hat{s}_n (x+iy) \right| \mathrm{d}x \mathrm{d}y .
\end{multline*}
By Lemma \ref{lem:psi}, uniformly in $x \in [x_0 - \varepsilon_n - 2 \omega_n, x_0 + \varepsilon_n + 2 \omega_n]$ and $y \in [-2,-1] \cup [1,2]$, $\left| \hat{s}_n (x+iy) \right| \leq \frac{n^\varepsilon}{\sqrt{n}}$. Therefore:
\begin{equation} \label{eq:bound3}
\left| \frac{1}{2 \pi} \int_{x \in \mathbf{R}} \int_{y \in \mathbf{R}} \left( \phi_n(x) + iy\phi'_n(x)  \right) \chi'(y) \hat{s}_n (x+iy) \mathrm{d}x \mathrm{d}y \right| = O \left( \frac{n^\varepsilon}{\sqrt{n}}  \right).
\end{equation}

\paragraph*{Conclusion}
Putting together estimates \eqref{eq:bound1}, \eqref{eq:bound2} and \eqref{eq:bound3}, we proved that, for all $\varepsilon>0$, on the event $E_\varepsilon$:
\begin{equation}
\left| \int_{\mathbf{R}} \phi_n(\lambda) \mathrm{d}\hat{\mu}_n(\lambda)   \right| = O\left( \max \left\{ \frac{n^{\varepsilon} }{n^{1/2} }, \omega_n \right\} \right) = O(\omega_n) .
\end{equation}
Now, recall from \eqref{eq:omega} that $\omega_n > \! \! > n^\delta / \sqrt{n}$. This implies that for all $0 <\varepsilon \leq \delta$:
\[ \max \left\{ \frac{n^{\varepsilon} }{n^{1/2} }, \omega_n \right\} = \omega_n.  \]
Hence, there exists $D>0$ such that with probability at least $1-n^{-D}$:
\[ \left| \int_{\mathbf{R}} \phi_n(\lambda) \mathrm{d}\hat{\mu}_n(\lambda)   \right| = O \left( \omega_n \right). \]
This ends the proof of the first part of Lemma \ref{lem:Helff}.

\bigskip

It remains to explain how to adapt the above argument to obtain the second estimate \eqref{eq:Helff2}. Since it is concerned with the {\it empirical} spectral measure $\mu_{S_n}$, we need a local law for this quantity, which is an analog of Theorem \ref{th:LocLawStieltjes} in this context, and is a consequence of the work of Knowles and Yin. Recall that 
\[ \mathcal{D}_n^{(\tau)}(x_0,c) = \left\{ z \in \mathbf{C}, \, x_0-c \leq E \leq x_0+c, \, n^{-1 + \tau} \leq \eta \leq \tau^{-1}    \right\} . \]
\begin{Theo}{\bf \cite[Theorem 3.22]{MR3704770}}
Let $s_{\mu_{S_n}}$ be the Stieltjes transform of $\mu_{S_n}$. Let $\tau > 0$. Then, uniformly in $z  \in \mathcal{D}_n^{(\tau)}(x_0,c)$, for any $\varepsilon >0$, there exists $D>0$ such that:
\[ \PP \left( \left| s_{\mu_{S_n}}(z) - s_{\mu_{\mathrm{MP},\alpha} \boxtimes \mu_\Sigma}(z) \right| \geq \frac{n^\varepsilon}{n \Im(z)} \right) \leq \frac{1}{n^D} .  \]
\end{Theo}
As before, we only discuss the case where $\epsilon = +$ and denote $\phi_n = \phi_n^+$. As in the case of the spectral measure, denoting $\tilde{\mu}_n := \mu_{(S_n,v_1^{(n)})} - \mu_{\alpha,\Sigma,\theta}$ and $\tilde{s}_n := s_{\mu_{S_n}} - s_{\mu_{\mathrm{MP},\alpha} \boxtimes \mu_\Sigma}$, the Helffer-Sj\"ostrand formula leads to:
\begin{align} 
\int_{\mathbf{R}} \phi_n(\lambda) \mathrm{d}\tilde{\mu}_n(\lambda) \leq &\phantom{+b} \frac{-1}{2 \pi} \int_{x \in \mathbf{R}} \int_{|y| \leq \omega_n} \phi''_n(x) y \chi (y) \Im \left(\tilde{s}_n (x+iy)\right) \mathrm{d}x \mathrm{d}y \label{eq:Term1prime} \\
&+ \frac{-1}{2 \pi} \int_{x \in \mathbf{R}} \int_{|y| \geq \omega_n} \phi''_n(x) y \chi (y) \Im \left(\tilde{s}_n (x+iy) \right) \mathrm{d}x \mathrm{d}y \label{eq:Term2prime} \\
&+ \left| \frac{1}{2 \pi} \int_{x \in \mathbf{R}} \int_{y \in \mathbf{R}} \left( \phi_n(x) + iy\phi'_n(x)  \right) \chi'(y) \tilde{s}_n (x+iy) \mathrm{d}x \mathrm{d}y \right| \label{eq:Term3prime} .
\end{align}
Now, each term \eqref{eq:Term1prime}, \eqref{eq:Term2prime} and \eqref{eq:Term3prime} can be treated in the same way as the previous terms \eqref{eq:Term1}, \eqref{eq:Term2} and \eqref{eq:Term3}. The only difference is that each occurence of the former error term $\psi_n(z)$ is now replaced by the error term of Theorem \ref{th:LocLawStieltjes}, namely $(n \Im(z))^{-1}$. 
\end{proof}

We underline that a weaker version of Theorem \ref{theo: square proj Wishart general case} (resp. Theorem \ref{theo: square proj Wigner general case}) can be obtained as long as a uniform estimation of $s_{(S_n,v_1^{(n)})}(z) - s_{\mu_{\alpha, \Sigma, \theta}}(z)$ is available for $z$ in a domain of the upper half-plane that is allowed to approach the real axis as $n$ becomes larger. Indeed, if
\[  \left| s_{(S_n,v_1^{(n)})}(z) - s_{\mu_{\alpha, \Sigma, \theta}}(z) \right| = O\left( \varepsilon_n \right), \] 
then, the Helffer-Sj\"ostrand argument that we developed during the proof of Theorem \ref{theo: square proj Wishart general case} yields a convergence of the averaged-square projection onto the direction of the spike for averaging windows of size $\varepsilon_n$, as long as $\varepsilon_n > \! \! > n^\delta / \sqrt{n}$ for some $\delta >0$. This limitation corresponds to the optimal rate in the local laws of Knowles and Yin \eqref{eq:loclaw}.

One natural question would be to weaken the assumption on the size $\varepsilon_n$ of the averaging window: do our Theorems \ref{theo: averaged explicit Wigner}, \ref{theo: averaged explicit Wishart}, \ref{theo: square proj Wigner general case} and \ref{theo: square proj Wishart general case} hold as soon as $\varepsilon_n = o(1)$? We believe that the answer is positive (see Figure \ref{fig: conjecture} for simulations). Morevoever, the results of \cite{MR3449395}, which state the convergence in law of properly rescaled {\it individual} square projection of eigenvectors associated to eigenvalues in the vicinity of the edge, suggest the following natural question: does such a convergence also holds in the bulk of the spectrum?  

\begin{figure}[ht!]
\centering
\begin{subfigure}{.4\textwidth}
  \centering
  \includegraphics[scale=0.4]{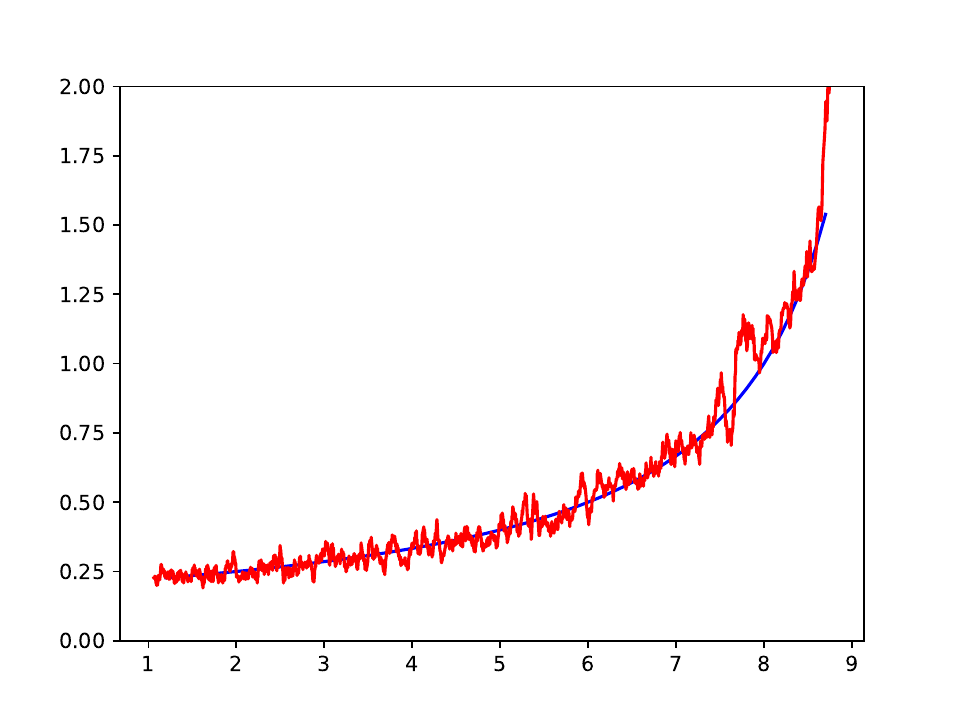} 
  \caption{}
  \label{fig: subfig1}
\end{subfigure}%
\begin{subfigure}{.4\textwidth}
  \centering
  \includegraphics[scale=0.4]{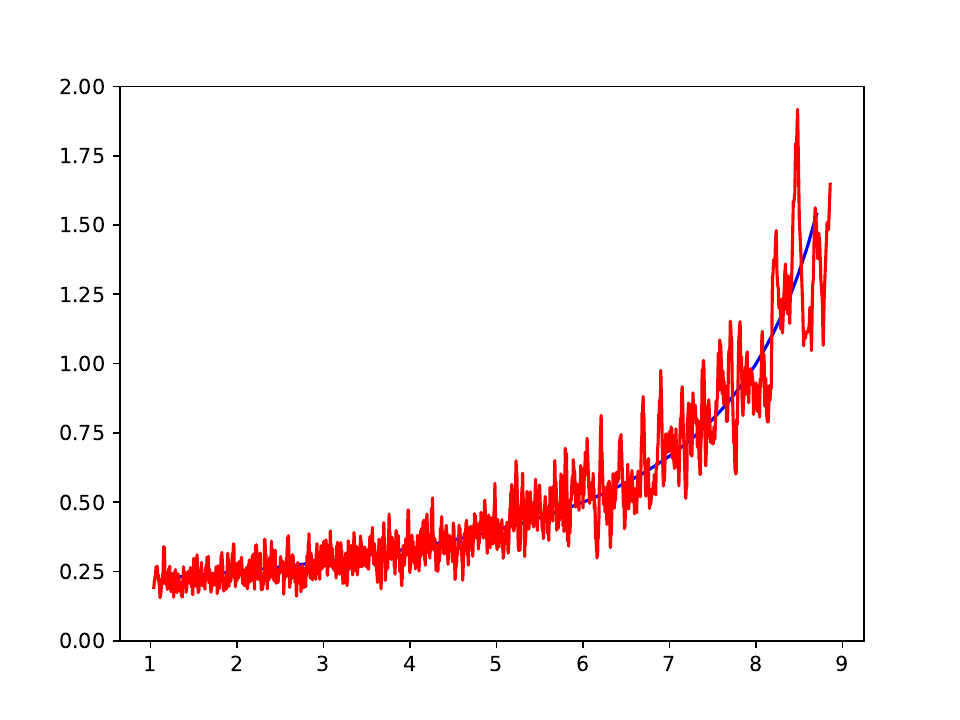} 
  \caption{}
  \label{fig: subfig2}
\end{subfigure}
\caption{{\footnotesize In red: simulations of the average squared projections around all locations $x\in((1-\sqrt{\alpha})^2,(1+\sqrt{\alpha})^2)$ where we took average over interval of typical size $n^{0,3}$ for \ref{fig: subfig1} and $n^{0,2}$ for \ref{fig: subfig2}, for $10$ independent matrices of the form $S_n = n^{-1} \mathrm{Diag}(\sqrt{\theta},1, \ldots, 1) X_nX_n^T \mathrm{Diag}(\sqrt{\theta},1, \ldots, 1)$ where $X_n$ is gaussian rectangular of size $2000 \times 8000$ ($\alpha=4$) in case \ref{fig:sfig11} and of size $3000 \times 12000$ ($\alpha=4$) in case \ref{fig:sfig21}. In each case $\theta = 2$. In blue: theoretical predictions.}}
\label{fig: conjecture}
\end{figure}

\paragraph*{Acknowledgment} I am very thankful to Nathana\"el Enriquez for many suggestions about this work. Many thanks to Maxime F\'evrier for numerous insightful discussions and for his careful reading of an earlier version of this paper. I would also like to thank Laurent M\'enard for his advices.

\bibliographystyle{plain}

\bigskip
\noindent Nathan Noiry :\\
Laboratoire Modal'X, \\
UPL, Université Paris Nanterre,\\ 
F92000 Nanterre France

\end{document}